\documentclass[a4paper,final]{amsart}

\usepackage{amsmath, amssymb, amsthm}
\usepackage{enumitem,fullpage}
\usepackage{color}
\usepackage{here}
\usepackage{tikz}
\usetikzlibrary{intersections, calc, arrows, cd}

\numberwithin{equation}{section}
\numberwithin{figure}{section}
\allowdisplaybreaks

\theoremstyle{definition}
\newtheorem{thm}{Theorem}[section]
\newtheorem{prp}[thm]{Proposition}
\newtheorem{lem}[thm]{Lemma}
\newtheorem{cor}[thm]{Corollary}
\newtheorem{dfn}[thm]{Definition}

\newtheorem{rmk}[thm]{Remark}
\newtheorem{ex}[thm]{Example}

\newtheorem{fct}[thm]{Fact}
\newtheorem*{thm*}{Theorem}
\newtheorem*{prp*}{Propostion}
\newtheorem*{lem*}{Lemma}
\newtheorem*{dfn*}{Definition}
\newtheorem*{nt*}{Notattion}
\newtheorem*{fct*}{Fact}
\newtheorem*{rmk*}{Remark}
\newtheorem*{ex*}{Example}

\newtheorem{thma}{Theorem}

\newcommand{\ol}{\overline}
\newcommand{\ul}{\underline}

\newcommand{\wt}{\widetilde}

\newcommand{\xr}[1]{\xrightarrow{\, #1 \, }}
\newcommand{\xl}[1]{\xleftarrow{\, #1 \, }}
\newcommand{\inj}{\hookrightarrow}
\newcommand{\surj}{\twoheadrightarrow}

\newcommand{\simto}{\xr{\sim}}
\newcommand{\simot}{\xl{\sim}}

\newcommand{\bbC}{\mathbb{C}}
\newcommand{\bbD}{\mathbb{D}}

\newcommand{\bbZ}{\mathbb{Z}}

\newcommand{\bbk}{\Bbbk}

\DeclareMathOperator{\cha}{char}
\DeclareMathOperator{\lcm}{lcm}

\newcommand{\abs}[1]{\left| #1 \right|}

\DeclareMathOperator{\Mod}{Mod}
\DeclareMathOperator{\catmod}{mod}

\DeclareMathOperator{\Gr}{GrMod}

\DeclareMathOperator{\CM}{CM}

\newcommand{\catA}{\mathcal{A}}
\newcommand{\catB}{\mathcal{B}}
\newcommand{\catC}{\mathcal{C}}
\newcommand{\catD}{\mathcal{D}}

\newcommand{\catI}{\mathcal{I}}
\newcommand{\catJ}{\mathcal{J}}
\newcommand{\catP}{\mathcal{P}}
\newcommand{\catT}{\mathcal{T}}
\newcommand{\catU}{\mathcal{U}}

\newcommand{\smod}{\mathop{\ul{\catmod}}\nolimits}

\newcommand{\sCM}{\ul{\CM}}

\DeclareMathOperator{\id}{id}
\DeclareMathOperator{\Id}{Id}

\DeclareMathOperator{\Ob}{Ob}

\DeclareMathOperator{\Hom}{Hom}
\DeclareMathOperator{\End}{End}
\DeclareMathOperator{\Ext}{Ext}
\DeclareMathOperator{\Ima}{Im}
\DeclareMathOperator{\Ker}{Ker}

\DeclareMathOperator{\Cone}{Cone}

\newcommand{\po}{\arrow[rd,"{\mathrm{PO}}",phantom]}
\newcommand{\pb}{\arrow[rd,"{\mathrm{PB}}",phantom]}

\DeclareMathOperator{\Proj}{Proj}
\DeclareMathOperator{\proj}{proj}

\DeclareMathOperator{\add}{add}

\DeclareMathOperator{\tria}{tria}
\DeclareMathOperator{\thick}{thick}

\DeclareMathOperator{\rad}{rad}

\newcommand{\sHom}{\ul{\Hom}}

\newcommand{\op}{\text{op}}

\newcommand{\pd}{\mathop{\mathrm{proj.dim}}\nolimits}

\newcommand{\gd}{\mathop{\mathrm{gl.dim}}\nolimits}

\newcommand{\dg}{\text{dg}}
\newcommand{\dgHom}{\mathcal{H}\text{om}\,}
\DeclareMathOperator{\perf}{perf}

\newcommand{\pmm}{\mathfrak{m}}
\newcommand{\pnn}{\mathfrak{n}}
\newcommand{\ppp}{\mathfrak{p}}

\newcommand{\LL}{\Lambda}

\newcommand{\Sg}{\Sigma}
\newcommand{\bu}{\bullet}

\newcommand{\arr}[1]{\arrow[{#1}]}

\newenvironment{enur}{\begin{enumerate}[nosep, label=(\roman*)]}{\end{enumerate}}
\newenvironment{enua}{\begin{enumerate}[nosep, label=(\arabic*)]}{\end{enumerate}}

\makeatletter
\newcommand{\colim@}[2]{%
	\vtop{\m@th\ialign{##\cr
			\hfil$#1\operator@font colim$\hfil\cr
			\noalign{\nointerlineskip\kern1.5\ex@}#2\cr
			\noalign{\nointerlineskip\kern-\ex@}\cr}}%
}
\newcommand{\colim}{%
	\mathop{\mathpalette\colim@{\rightarrowfill@\textstyle}}\nmlimits@
}
\makeatother

\begin{document}

\title{Tilting objects in periodic triangulated categories}
\author{Shunya SAITO}
\date{June 6, 2021}
\address{Graduate School of Mathematics, Nagoya University, Chikusa-ku, Nagoya. 464-8602, Japan}
\email{m19018i@math.nagoya-u.ac.jp}
\maketitle

\begin{abstract}
A triangulated category $\catT$ whose suspension 
functor $\Sg$ satisfies $\Sg^m \simeq \Id_{\catT}$ as additive functors
is called an $m$-periodic triangulated category.
Such a category does not have a tilting object by the periodicity.
In this paper, we introduce the notion of an $m$-periodic tilting object
in an $m$-periodic triangulated category,
which is a periodic analogue of a tilting object in a triangulated category,
and prove that an $m$-periodic triangulated category 
having an $m$-periodic tilting object is triangulated equivalent to
the $m$-periodic derived category of an algebra
under some homological assumptions.
As an application, we construct a triangulated equivalence
between the stable category of a self-injective algebra 
and the $m$-periodic derived category of a hereditary algebra.
\end{abstract}

\tableofcontents

\section{Introduction}\label{s:Intro}
\subsection{Background}

\emph{Tilting theory} is a certain generalization of Morita theory,
describing a way to relate a triangulated category
with the derived category of an algebra.
It plays an important role in representation theory
and gives connections between various areas of mathematics.
For example, it appears in a classification of representations of Dynkin quivers,
Cohen-Macaulay representations of Gorenstein singularities and 
commutative/non-commutative algebraic geometry.
See \cite{Tilt07} for a broad review of the theory.
The aim of this paper is to give a starting point of 
\emph{tilting theory for periodic triangulated categories}.

Let $m$ be a positive integer.
A triangulated category $\catT$ whose suspension functor $\Sg$ satisfies 
$\Sg^m \simeq \Id_{\catT}$ as additive functors
is called an \emph{$m$-periodic triangulated category}.
A basic example of a periodic triangulated category
is an \emph{$m$-periodic derived category},
the localization of the category of $m$-periodic complexes 
with respect to quasi-isomorphisms.
See \S \ref{s:m-cpx} for the precise definition.
It appears in the categorification of Lie algebras and quantum groups 
via Ringel-Hall algebras:
Peng and Xiao \cite{PX97,PX00} used two-periodic complexes to 
construct full semisimple Lie algebras.
Bridgeland \cite{Br13} investigated the two-periodic derived category 
of a hereditary abelian category and constructed full quantum groups.
Motivated by these studies,
several authors analyzed the structure of $m$-periodic derived categories.
Peng and Xiao proved in \cite{PX97} that
the $m$-periodic derived category $D_m(\LL)$ of a hereditary algebra $\LL$
is triangulated equivalent to the orbit category $D^b(\LL)/\Sg^m$ 
of the bounded derived category $D^b(\LL)$.
In \cite{Zhao14}, Zhao generalized this result to an algebra $\LL$ 
of finite global dimension, and proved that
the $m$-periodic derived category $D_m(\LL)$ is triangulated equivalent to 
the triangulated hull of the orbit category $D^b(\LL)/\Sg^m$.
In particular, the $m$-periodic derived category is a derived invariant of algebras.
See also \cite{Fu12,Go13,Go18}.
We will give a relevant explanation in \S \ref{s:m-cpx}.

Other examples of periodic triangulated categories are given by 
the stable category of maximal Cohen-Macaulay modules over a hypersurface singularity,
and the stable category of finitely generated modules over 
a self-injective algebra of finite representation type.
We will treat some of these examples in \S \ref{s:ex}.

In the ordinary tilting theory, we have the \emph{tilting theorem},
which states that a triangulated category with a tilting object 
is triangulated equivalent to the perfect derived category of
the endomorphism ring of the tilting object.
See Fact \ref{fct:Keller} for the precise statement.
The motivation of this paper is to give
a \emph{periodic analogue of this tilting theorem}.
The main Theorem \ref{thm:A} gives a sufficient condition 
under which a periodic triangulated category
is triangulated equivalent to the periodic derived category of an algebra.

\subsection{Main results}
Let $\catT$ be a triangulated category 
with the suspension functor $\Sg$ over a field $\bbk$.

We first recall the tilting objects and the tilting theorem 
in the standard sense.
As usual, a \emph{thick subcategory} of a triangulated category means 
a full subcategory closed under taking cones, shifts and direct summands.
An object of a triangulated category is a \emph{thick generator}
if the smallest thick subcategory containing $T$ is equal to $\catT$.
\begin{dfn*}
An object $T$ of $\catT$ is called \emph{tilting}
if it is a thick generator of $\catT$ with $\catT(T,\Sg^iT)=0$ for any $i\ne0$.
\end{dfn*}

A triangulated category is called algebraic
if it is triangulated equivalent to 
the stable category of a Frobenius exact category.
\begin{fct}[{The tilting theorem \cite{Kel94}}]\label{fct:Keller}
Let $\catT$ be an idempotent complete algebraic triangulated category.
If $\catT$ has a tilting object $T$,
then there exists a triangulated equivalence
\[
\catT \simto K^b(\proj \LL),
\]
where $\LL :=\End_{\catT}(T)$ is the endomorphism ring of $T$,
and $K^b(\proj \LL)$ is the homotopy category 
of bounded complexes of projective $\LL$-modules.
\end{fct}

An $m$-periodic triangulated category does not have a tilting object
by the periodicity,
and cannot be triangulated equivalent to 
the derived category of an algebra.
Instead, we will introduce an \emph{$m$-periodic tilting object}
which is a periodic analogue of a tilting object, 
and prove a periodic analogue of the tilting theorem.

\begin{dfn}\label{dfn:m-tilt}
Let $\catT$ be an $m$-periodic triangulated category.
An object $T$ of $\catT$ is an \emph{$m$-periodic tilting} object
if it is a thick generator of $\catT$ with $\catT(T,\Sg^iT)=0$ for any $i\not\in m\bbZ$.
\end{dfn}

\begin{thma}[{Theorem \ref{thm:Main}}]\label{thm:A}
Let $\catT$ be an idempotent complete algebraic 
$m$-periodic triangulated category having an $m$-periodic tilting object $T$, 
and $\LL :=\End_{\catT}(T)$ the endomorphism algebra of $T$ in $\catT$.
Suppose $\LL$ is a finite dimensional $\bbk$-algebra and satisfies
\begin{align}\label{eq:thmA:cond}
 \pd_{\LL^{\op}\otimes \LL}\LL \le m.
\end{align}
Then there exists a triangulated equivalence
\[
\catT\simeq D_m\left( \Lambda \right),
\]
where 
$D_m\left( \Lambda \right)$ is the $m$-periodic derived category of $\LL$
(\S \ref{s:m-cpx}).
\end{thma}

At the first look, one may worry about the condition \eqref{eq:thmA:cond} 
of projective dimension as bimodules. 
In fact, we can replace it to the condition of global dimension of algebras
under a condition on the base field.
\begin{thma}[{Corollary \ref{cor:Main}}]\label{thm:B}
Let $\catT$ be an idempotent complete algebraic
$m$-periodic triangulated category over a perfect field $\bbk$.
If $\catT$ has an $m$-periodic tilting object $T$
which satisfies the condition
\[
\text{$\LL :=\End_{\catT}(T)$
is a finite dimensional $\bbk$-algebra of global dimension $d\leq m$},
\]
then there exists a triangulated equivalence
\[
\catT\simeq D_m\left( \Lambda \right).
\]
\end{thma}

A key statement to prove Theorem \ref{thm:A} is Theorem \ref{thm:C},
which states that the Laurent polynomial ring is 
\emph{intrinsically formal} (Definition \ref{dfn:dg:formal}) under some assumptions.

\begin{thma}[{Corollary  \ref{cor:formal}}]\label{thm:C}
Let $\LL$ be an $\bbk$-algebra which has a finite projective dimension $d$ 
as a $\LL$-bimodule.
We regard the Laurent polynomial ring $\LL[t,t^{-1}]$ over $\LL$ 
as a graded algebra by setting $\deg(t):=m$.
If $d \leq m$, then $\LL[t, t^{-1}]$ is intrinsically formal.
\end{thma}

The proof of this statement is inspired by \cite{HK19}.

\subsection*{Organization}

This paper is organized as follows.
In Section \ref{s:pre},
we recall some basic facts on DG categories and 
Hochschild cohomology which will be used in the proof of Theorem \ref{thm:A}.
In Section \ref{s:m-cpx}, 
we give a detailed survey of periodic derived categories.
We emphasis that differential graded categorical point of view.
In Section \ref{s:main},
we give a proof of main theorems.
In Section \ref{s:ex},
we give concrete examples of periodic categories and periodic tilting objects.
We give an application of Theorem \ref{thm:B} 
and an example that Theorem \ref{thm:B} fails 
for an algebra of infinite global dimension.
\subsection*{Acknowledgement}
The author would like to thank his supervisor Shintaro Yanagida
for valuable suggestions and discussion.
He is also very grateful to Norihiro Hanihara.
The author learned tilting theory for triangulated categories 
and DG categories from him,
and discussions with him played a crucial role
in the early stages of this project.
He would like to thank Osamu Iyama, Ryo Takahashi and Haruhisa Enomoto
for insightful comments and suggestions.

\subsection*{Conventions and notation}
\begin{itemize}[nosep]
\item 
$\bbZ_m$ denotes the cyclic group of order $m \in \bbZ_{>0}$.

\item 
Throughout this paper we work over a fixed field $\bbk$.
$\cha(\bbk)$ denotes its characteristic.
All categories and functors are $\bbk$-linear.
All subcategories are full subcategories which is closed under isomorphisms.

Let $\catC$ be a category.
\begin{itemize}[nosep]
\item
$\Ob \catC$ denotes the class of objects of $\catC$.
\item
We denote by $\catC(M,N) = \Hom_{\catC}(M,N)$ the set of morphisms
from $M$ to $N$ for $M, N \in\catC$. 
\end{itemize}

\item 
We denote by $\Sg$ the suspension functor of a triangulated category.

Let $\catI$ be a collection of objects in a triangulated category $\catT$.
\begin{itemize}[nosep]
\item
$\add_{\catT}(\catI)$ denotes
the smallest strictly full additive subcategory of $\catT$ 
which contains $\catI$ and is closed under direct summands.
\item
$\tria_{\catT}(\catI)$ denotes
the smallest strictly full triangulated subcategory of $\catT$ 
which contains $\catI$.
\item
$\thick_{\catT}(\catI)$ denotes
the smallest strictly full triangulated subcategory of $\catT$ 
which contains $\catI$ and is closed under direct summands.
\end{itemize}
\item 
By the word ``algebra" we mean an associative unital finite dimensional algebra
over the base field $\bbk$.
By the word ``module" we mean a right module.

Let $\LL$ be an algebra.
\begin{itemize}[nosep]
\item
$\rad \LL$ denotes the Jacobson radical of $\LL$.
\item
$\gd \LL$ denotes the global dimension of $\LL$.
\item
For a $\LL$-module $M$,
$\pd_{\LL} M$ denotes the projective dimension of $M$.
\item
$\Mod\LL$ (resp.\ $\catmod\LL$) denotes
the category of all (resp.\ finitely generated) $\LL$-modules.
\item
$\Proj \LL$ (resp.\ $\proj\LL$) denotes
the category of all (resp.\ finitely generated) projective $\LL$-modules.
\end{itemize}

\item
By the words ``graded algebra" and ``graded module",
we mean $\bbZ$-graded ones unless otherwise stated.
For a graded module $M$,
we denote the grading structure by $M=\bigoplus_{i \in \bbZ} M^i$,
and the degree of an element $m\in M$ by $\deg(m)$ or $\abs{m}$.

Let $A$ be a graded algebra.
\begin{itemize}[nosep]
\item
$\Gr A$ denotes the category of graded $A$-modules,
i.e., the category of graded $A$-modules and homogeneous $A$-linear maps of degree $0$.
\item
For $\ell \in\bbZ$, $(\ell)$ denotes the $\ell$-degree shift, 
i.e., $M(\ell)^i=M^{i+\ell}$ for $M \in \Gr A$.
\end{itemize}
\end{itemize}

\section{Preliminaries}\label{s:pre}
\subsection{DG categories}\label{ss:dg}
In this subsection,
we recall basic notions on differential graded (DG) categories.
See \cite{Kel94,Dr04} for more details.
The tensor product over $\bbk$ will be denoted by $\otimes$.

A DG $\bbk$-module $V=(V,d_V)$ is a complex of $\bbk$-vector spaces.
The tensor product of two DG $\bbk$-modules $V$ and $W$ is
the graded $\bbk$-module $V\otimes W$ with differential defined by
\[
d_{V\otimes W}(v\otimes w)=d_V(v)\otimes w +(-1)^iv\otimes d_W(w),
\]
where $v\in V^i$ and $w\in W^j$.
A DG category is a $\bbk$-linear category $\catA$
whose morphism spaces $\catA(X,Y)$ are DG $\bbk$-modules 
and whose compositions
\[
\catA(Y,Z)\otimes\catA(X,Y)\to \catA(X,Z),\quad g\otimes f \mapsto gf
\]
are morphism of DG $\bbk$-modules,
that is, for all homogeneous elements $f\in\catA(X,Y)^i$ and $g\in\catA(Y,Z)^j$,
the graded Leibniz rule
\[
d_{\catA(X,Z)}(gf)=d_{\catA(Y,Z)}(g)f + (-1)^jgd_{\catA(X,Y)}(f)
\]
holds.
$f\in\catA(X,Y)^p$ is called a morphism of degree $p$.
$Z^p(\catA)(X,Y):=Z^p(\catA(X,Y))$ denotes
the $p$th cocycle of a DG $\bbk$-module $\catA(X,Y)$.
Similarly,
$B^p(\catA)(X,Y):=B^p(\catA(X,Y))$ denotes
the $p$th coboundary of a DG $\bbk$-module $\catA(X,Y)$.
A morphism $f:X\to Y$ is \emph{closed} 
if $f\in Z(\catA)(X,Y):=\bigoplus_{p\in\bbZ}Z^p(\catA)(X,Y)$.
Also, $f$ is \emph{exact}
if $f\in B(\catA)(X,Y):=\bigoplus_{p\in\bbZ}B^p(\catA)(X,Y)$.

Let $\catA$ be a (small) DG category.
We define categories $Z^0(\catA)$ and $H^0(\catA)$.
$Z^0(\catA)$ has the same objects as $\catA$ and 
its morphism spaces are $Z^0(\catA)(X,Y)$.
Similarly, $H^0(\catA)$ has the same objects as $\catA$ 
and its morphism spaces are defined by
\[
H^0(\catA)(X,Y) := H^0\left(\catA(X,Y)\right), 
\]
the $0$th cohomology of a DG $\bbk$-module $\catA(X,Y)$.
We call $H^0(\catA)$ the \emph{homotopy category} of $\catA$.
Note that $H^0(\catA)$ is not necessarily a triangulated category.

\begin{ex}
The DG category $C_{\dg}(\bbk)$ of DG $\bbk$-modules
is defined as follows.
\begin{itemize}[nosep]
\item
The objects are DG $\bbk$-modules.
\item 
The morphism spaces are DG $\bbk$-modules $\dgHom_{\bbk}(V,W)$,
where $\dgHom_{\bbk}(V,W)^p$ is the $\bbk$-vector space of
homogeneous morphisms $f:V\to W$ of degree $p$ and
\begin{align*}
d_{\dgHom_{\bbk}(V,W)}(f) := d_W\circ f - (-1)^p f\circ d_V.
\end{align*}
\end{itemize}
Then
$C(\bbk):=Z^0(C_{\dg}(\bbk))$ is
the category of complexes over $\bbk$-vector spaces and chain maps.
$K(\bbk):=H^0(C_{\dg}(\bbk))$ is 
the usual homotopy category of complexes over $\bbk$-vector spaces.
\end{ex}

\begin{ex}
Let $\catA$ be a DG category.
The \emph{opposite DG category} $\catA^{\op}$ consists of the same objects as $\catA$,
and morphism spaces $\catA^{\op}(X,Y):=\catA(Y,X)$.
The composition of $\catA^{\op}$ is given by
\[
\catA^{\op}(Y,Z)\otimes \catA^{\op}(X,Y) \to \catA^{\op}(X,Z),
\quad g\otimes f \mapsto (-1)^{\abs{f}\abs{g}}fg.
\]
\end{ex}

A DG functor between DG categories $\catA$ and $\catB$
is a functor $F:\catA \to \catB$ with
\[
F_{X,Y}:\catA(X,Y)\to \catB(FX,FY)
\]
is a morphism of DG $\bbk$-modules for all $X,Y\in\catA$.

Let $\catA$ be a DG category.
A (right) DG $\catA$-module is
a DG functor $\catA^{\op} \to C_{\dg}(\bbk)$.
We denote by $C_{\dg}(\catA)$ the DG category of DG $\catA$-modules.
See \cite[\S 1]{Kel94} for the definition of morphism spaces of $C_{\dg}(\catA)$.
Each object $X \in \catA$ produces a DG $\catA$-module
\[
X^{\wedge} := \catA(-, X),
\]
which is called a \emph{representable} DG $\catA$-module.
This gives a fully faithful DG functor $\catA \inj C_{\dg}(\catA)$,
which is called the \emph{DG Yoneda embedding}.
We write $C(\catA):=Z^0\left( C_{\dg}(\catA) \right)$
and $K(\catA):= H^0\left( C_{\dg}(\catA) \right)$. 
$C(\catA)$ is a Frobenius exact category.
The stable category of $C(\catA)$
coincides with the homotopy category $K(\catA)$.
In particular, $K(\catA)$ is an algebraic triangulated category.
The \emph{derived category} $D(\catA)$ of $\catA$ is
the localization of $K(\catA)$ with respect to quasi-isomorphisms.
The \emph{perfect derived category} of $\catA$ is defined by
\[
\perf(\catA):= \thick_{D(\catA)} \left( \text{representable DG $\catA$-modules} \right)
\]
and call an object of $\perf(\catA)$ a \emph{perfect} DG $\catA$-module.

We can define the shift of a DG $\catA$-module,
and the cone of a morphism in $C(\catA)$
in a similar manner to complexes.
Note that $\catA$ can be viewed as a DG subcategory of $C_{\dg}(\catA)$
by DG Yoneda embedding.
$\catA$ is said to be an \emph{exact} DG category
if $\catA$ has the zero object and,
is closed under taking shifts and cones in $C_{\dg}(\catA)$.
In this case, 
$Z^0(\catA)$ is a Frobenius exact category
whose stable category is $H^0(\catA)$,
and thus it is a triangulated category.
A fully faithful functor $H^0(\catA) \to K(\catA)$
induced by the DG Yoneda embedding is a triangulated functor.
When a triangulated category $\catT$ is triangulated equivalent to
$H^0(\catA)$ for some exact DG category $\catA$,
we say that $\catA$ is a \emph{DG enhancement} of $\catT$.

We can characterize exact DG categories intrinsically.
For $X\in\catA$ and $\ell\in \bbZ$,
$Z\in\catA$ is the $\ell$th \emph{shift} of $X$
if there are closed morphisms $\xi_X :Z\to X$ and $\zeta_X :X\to Z$
of degree $\ell$ and $-\ell$ respectively,
which satisfies $\zeta_X\xi_X=\id_{Z}$ and $\xi_X\zeta_X=\id_X$.
Let $f:X \to Y$ be a closed morphism of degree $0$
and suppose that $X$ has the $1$st shift $X[1]$ of $X$.
$\Cone(f)\in \catA$ is the \emph{cone} of $f$
if there exists a diagram
\begin{equation*}\label{diag:cone}
\begin{tikzcd}
Y \arrow[r,yshift=0.7ex,"{i_f}"] & \Cone(f) \arrow[r,yshift=0.7ex,"{p_{f}}"] \arrow[l,yshift=-0.7ex,"{q_{f}}"] & X[1] \arrow[l,yshift=-0.7ex,"{j_{f}}"],
\end{tikzcd}
\end{equation*}
which is called the \emph{cone diagram} of $f$,
satisfying
\begin{itemize}[nosep]
\item $i_f, j_f, p_f, q_f$ are morphisms of degree $0$.
\item $q_f\circ i_f=\id_{Y}$, $p_f\circ j_f=\id_{X[1]}$,
$i_f\circ q_f + j_f\circ p_f = \id_{\Cone(f)}$.
\item $d(i_f)=0$, $d(p_f)=0$, $d(j_f)=i_f\circ  f\circ \xi_X$,
$d(q_f)=-f\circ \xi_X\circ p_f$. 
\end{itemize}
If every object has its $\pm 1$st shifts and 
every closed morphism of degree $0$ has its cone in $\catA$,
then $\catA$ is an exact DG category.

We denote by $\catA^0$ the additive category consisting of the same object as $\catA$
and its morphism spaces are defined by 
$\catA^0(X,Y):=\catA(X,Y)^0$, the $0$th part of a graded module $\catA(X,Y)$. 
A sequence $X\to Y \to Z$ in $Z^0(\catA)$ is \emph{graded split}
if it is a split exact sequence in $\catA^0$.
An object $X\in\catA$ is called \emph{contractible}
if its identity morphism $\id_X$ is an exact morphism.
Suppose $\catA$ is an exact DG category.
We can describe the exact structure on $Z^0(\catA)$ by the above notions.
A sequence $X\to Y \to Z$ in $Z^0(\catA)$ is 
a conflation of a Frobenius exact category $Z^0(\catA)$
if and only if it is a graded split sequence.
$X\in Z^0(\catA)$ is a projective object of a Frobenius exact category $Z^0(\catA)$
if and only if it is contractible.

We identify a DG algebra with a DG category with one object.
Hence the notation above is also used for DG algebras.
For a DG algebra $A$,
a graded algebra $H^*(A):=\bigoplus_{i\in\bbZ}H^i(A)$
is called the \emph{cohomology algebra} of $A$.
A morphism $f:A\to B$ of DG algebras is said to be \emph{quasi-isomorphism}
if induced morphism $H^*(f):H^*(A) \to H^*(B)$ on their cohomology algebras
is an isomorphism of graded algebras.
Two DG algebras $A$ and $B$ are \emph{quasi-isomorphic}
if there exist DG algebras $C_1,\dots, C_n$ and 
a chain of quasi-equivalence $\catA \simot C_1 \simto \cdots \simot C_n \simto B$.

We recall some facts which we use in this paper.
\begin{fct}[{\cite[\S 6.1]{Kel94}}]\label{fct:DG qis}
If two DG algebras $A$ and $B$ are quasi-isomorphic,
there exists a triangulated equivalence $D(A) \simeq D(B)$.
In particular, this equivalence induces a triangulated equivalence
$\perf(A) \simeq \perf(B)$.
\end{fct}

\begin{fct}[{\cite[\S 4.3]{Kel94}}]\label{fct:morita}
Let $\catT$ be an idempotent complete algebraic triangulated category, 
and $T\in\catT$.
Then there exists a DG algebra $B$ satisfying the following conditions.
\begin{enur}
\item
There exists a triangulated equivalence
$\thick_{\catT}(T)\simeq \perf(B) \subset D(B)$.
\item
$H^*(B):=\bigoplus_{i\in\bbZ}H^i(B)\simeq \bigoplus_{i\in\bbZ}\Hom_{\catT}(T, \Sg^i T)$
as graded algebras.
\end{enur}
\end{fct}

Hence the triangulated structure of $\thick_{\catT}(T)$ reflects
the structure of DG algebra $B$.
Under some assumption, 
DG algebras are determined by its cohomology ring $H^*(B)$.
\begin{dfn}\label{dfn:dg:formal}
Let $B$ be a DG algebra.
\begin{enua}
\item $B$ is \emph{formal} if it is quasi-isomorphic to $H^*(B)$.
\item A graded algebra $A$ is \emph{intrinsically formal}
if any DG algebra $B$ satisfying $H^*(B)\simeq A$ as graded algebras
is formal.
\end{enua}
\end{dfn}

\subsection{Hochschild cohomology}\label{ss:Hoc}
To describe a sufficient condition of being intrinsically formal,
we cite from \cite{RW11} a fact on Hochschild cohomology.
See \cite[\S 9]{W94} for the basics and the detail on Hochschild cohomology.

Let $A$ be a graded algebra.
We denote by $A^e := A^{\op} \otimes A$ the \emph{enveloping algebra} of $A$.
Thus $A^e$-modules are nothing but $A$-$A$-bimodules such that $\bbk$ acts centrally.

\begin{dfn}
The \emph{bar resolution} of $A$ is the complex of graded $A^e$-modules
\[
B_{\bu}: \cdots \to A^{\otimes (q+2)} \xr{d_q} A^{\otimes (q+1)}
\to \cdots \to A^{\otimes 2} \to 0,
\]
where $B_q:=A^{\otimes (q+2)}$ and $d_q$ is given by
\[
 d_q(a_1\otimes \dots \otimes a_{q+2})
 := \sum_{i=1}^{q+1}(-1)^{i} a_1\otimes \dots \otimes (a_i a_{i+1}) \otimes a_{q+2}.
\]
\end{dfn}

\begin{dfn}
Let $M$ be a graded $A$-module.
The \emph{Hochschild cohomology} of $A$ valued in $M$ is defined by
\[
 HH^{p,q}(A,M) := H^p\left(\Hom_{\Gr A^e}\left(B_{\bu}, M(q)\right)\right)
\]
for $p,q \in \bbZ$.
To shorten the notation, we write $HH^{p,q}(A)$ instead of $HH^{p,q}(A,A)$.
\end{dfn}
\begin{rmk}
Since the bar resolution of $A$ is
a projective resolution of $A$ as a graded $A^e$-module,
we have $HH^{p,q}(A,M) = \Ext^p_{\Gr A^e}(A, M(q))$.
Hence for any other projective resolution $P_{\bu}$ of $A$ as a graded $A^e$-module,
we have $HH^{p,q}(A, M) =H^p\left(\Hom_{\Gr A^e}\left(P_{\bu}, M(q)\right)\right)$.
\end{rmk}

The next fact gives a useful criterion for a graded algebra
to be intrinsically formal.
\begin{fct}[{\cite[Corollary 1.9]{RW11}}]\label{fct:if}
Let $A$ be a graded algebra.
If $HH^{q,2-q}(A)=0$ for $q \geq 3$,
then $A$ is intrinsically formal.
\end{fct}
\section{$m$-periodic complexes}\label{s:m-cpx}
In this section, we give a detailed explanation of periodic complexes.
Some definitions and results are direct analogue of usual complexes,
but we write them down explicitly since we cannot find systematic literature.
We recommend the reader to only skim Definition \ref{dfn:DG m-cpx}, \ref{dfn:periodic acyclic} and \ref{dfn:periodic derived},
and go to \S \ref{ss:periodic cpx as DG} in the first reading.
References for this section are \cite{Br13,Go13,Zhao14,St18}.

Hereafter let $m$ be an integer greater than $1$.
For any integer $i\in\bbZ$,
the symbol $\ol{i} \in \bbZ_m$ is the equivalence class containing $i$.

\subsection{The DG category of periodic complexes}\label{ss:DG periodic cpx}
We give a summary of this subsection.
We first define the DG category $C_m(\catC)_{\dg}$
of periodic complexes over an additive category $\catC$ (Definition \ref{dfn:DG m-cpx})
and prove it is an exact DG category (Corollary \ref{cor:exact DG periodic}).
In particular,
its homotopy category $K_m(\catC):=H^0(C_m(\catC)_{\dg})$ is a triangulated category. 
For an abelian category $\catB$,
we define the $m$-periodic derived category $D_m(\catB)$ by
the localization of $K_m(\catB)$ with respect to quasi-isomorphisms
(Definition \ref{dfn:periodic derived}).
In the next subsection, we will see that $D_m(\catB)$ behaves similarly to
the bounded derived category if $\catB$ have finite global dimension.

\begin{dfn}\label{dfn:DG m-cpx}
Let $\catC$ be an additive category.
\begin{enua}
\item
The DG category $C_m(\catC)_{\dg}$ of \emph{$m$-periodic complexes} over $\catC$
consists of the following data.
\begin{itemize}[nosep]
\item
The objects
are families $V = (V,d_V) = (V^i, d_V^i)_{i\in\bbZ_m}$,
where $V_i \in \catC$ and  $d_V^i:V^i \to V^{i+1}$ is a morphism in $\catC$
satisfying $d_V^{i+1}d_V^i=0$ for all $i \in \bbZ_m$.
\item 
The morphism spaces
are DG $\bbk$-modules $C_m(\catC)_{\dg}(V,W)$,
where
\begin{align*}
&C_m(\catC)_{\dg}(V,W)^p := \bigoplus_{i\in\bbZ_m} \catC(V^i , W^{i+p}), \\
&d_{C_m(\catC)_{\dg}(V,W)}(f) := d_W\circ f - (-1)^p f\circ d_V 
\ (f \in C_m(\catC)_{\dg}(V,W)^p).
\end{align*}
\end{itemize}
\item
$C_m(\catC):=Z^0(C_m(\catC)_{\dg})$ is called
the category of $m$-periodic complexes over $\catC$.
\item
$K_m(\catC) := H^0C_m(\catC)_{\dg}$ is called
the \emph{homotopy category} of $m$-periodic complexes over $\catC$.
\end{enua}
\end{dfn}

An $m$-periodic complex $V$ is called a \emph{stalk complex}
if there exists $i_0\in\bbZ_m$ such that $V^{i_0} \ne 0$
and $V^i =0$ for all $i\ne i_0$.
In this case, we also say that $V$ is \emph{concentrated in degree $i_0$}.
Abusing the notation, we denote by $M\in C_m(\catC)_{\dg}$ 
a stalk complex whose $0$th component is $0\ne M\in\catC$.

\begin{ex}
Let $\catC$ be an additive category,
and $f: V \to W$ be a morphism of degree $0$ in $C_m(\catC)_{\dg}$.
If $f$ is a closed morphism,
\[
d_W\circ f = f \circ d_V
\]
by the definition of differential.
Thus we call $f$ a \emph{chain map}.
If $f$ is an exact morphism,
there exists $g\in C_m(\catC)_{\dg}(V,W)^{-1}$
such that $d(g) = f$.
By definition of the differential,
we have
\[
f=d_W\circ g + g \circ d_V.
\]
Hence we call two chain maps $f, g:V\to W$ are \emph{homotopic}
if $f$ and $g$ coincide in $K_m(\catC)$.
\end{ex}

We will prove that $C_m(\catC)_{\dg}$ is an exact DG category,
and hence $K_m(\catC)$ is a triangulated category.
\begin{prp}
Let $\catC$ be an additive category,
and $f:V\to W$ be a chain map in $C_m(\catC)$.
\begin{enua}
\item
For $\ell\in\bbZ$,
define
\[
V(\ell)^i:=V^{i+\ell},\quad
d^i_{V[\ell]}:=(-1)^{\ell}d^{i+\ell}_V.
\]
Then $V[\ell]:=\left(V(\ell),d_{V[\ell]}\right)$ is the $\ell$th shift of $V$ 
in the DG category $C_m(\catC)_{\dg}$.
\item
Define
\[
 C_f^i := W^i \oplus V^{i+1}, \quad 
 d_f^i := \begin{bmatrix} d_W^i & f^{i+1} \\ & -d_{V[1]}^{i+1}\end{bmatrix}.
\]
Then $\Cone(f):=\left(C_f,d_f\right)$ is the cone of $f$
in the DG category $C_m(\catC)_{\dg}$.
\end{enua}
\end{prp}
\begin{proof}
(1)
Note that $C_m(\catC)_{\dg}(V[\ell],W)\simeq C_m(\catC)_{\dg}(V,W)[-\ell]$.
\[
\xi_X :=\id_V\in C_m(\catC)(V,V)
=Z^0\left(C_m(\catC)_{\dg}\left(V[\ell],V\right)[\ell]\right)
=Z^{\ell}\left(C_m(\catC)_{\dg}\right)\left(V[\ell],V\right),
\]
is a desired morphism.

(2)
The natural injections and projections
\[
\begin{tikzcd}
W^i \arrow[r,yshift=0.7ex,"{i_f^i}"] & W^i\oplus V^{i+1} \arrow[r,yshift=0.7ex,"{p_{f}^i}"] \arrow[l,yshift=-0.7ex,"{q_{f}^i}"] & V^{i+1} \arrow[l,yshift=-0.7ex,"{j_{f}^i}"]
\end{tikzcd}
\]
define
morphisms
\[
\begin{tikzcd}
W \arrow[r,yshift=0.7ex,"{i_f}"] & \Cone(f) \arrow[r,yshift=0.7ex,"{p_{f}}"] \arrow[l,yshift=-0.7ex,"{q_{f}}"] & V[1] \arrow[l,yshift=-0.7ex,"{j_{f}}"]
\end{tikzcd}
\]
of degree $0$.
By the definition, we have
\[
q_f\circ i_f=\id_{W},\quad
p_f\circ j_f=\id_{V[1]},\quad
i_f\circ q_f + j_f\circ p_f = \id_{\Cone(f)}.
\]
In the following calculation, we omit the subscription $i\in\bbZ_m$.
\begin{align*}
d(i_f)
&=d_{\Cone(f)}\circ i_f - i_f\circ d_W\\
&=
\begin{bmatrix}
d_W & f\\
&-d_V
\end{bmatrix}
\begin{bmatrix}
1\\
0
\end{bmatrix}
-
\begin{bmatrix}
1\\
0
\end{bmatrix}
d_W
=
\begin{bmatrix}
d_W\\
0
\end{bmatrix}
-
\begin{bmatrix}
d_W\\
0
\end{bmatrix}
=0.\\
d(p_f)
&=d_{V[1]}\circ p_f - p_f\circ d_{\Cone(f)}\\
&=-d_V
\begin{bmatrix}
0&1
\end{bmatrix}
-
\begin{bmatrix}
0&1
\end{bmatrix}
\begin{bmatrix}
d_W & f\\
&-d_V
\end{bmatrix}
=
\begin{bmatrix}
0&-d_V
\end{bmatrix}
-
\begin{bmatrix}
0&-d_V
\end{bmatrix}
=0.\\
d(j_f)
&=d_{\Cone(f)}\circ j_f - j_f\circ d_{V[1]}\\
&=
\begin{bmatrix}
d_W & f\\
&-d_V
\end{bmatrix}
\begin{bmatrix}
0\\
1
\end{bmatrix}
+
\begin{bmatrix}
0\\
1
\end{bmatrix}
d_{V}
=
\begin{bmatrix}
f\\
-d_V
\end{bmatrix}
+
\begin{bmatrix}
0\\
d_V
\end{bmatrix}
=
\begin{bmatrix}
f\\
0
\end{bmatrix}\\
&=i_f\circ f.\\
d(q_f)
&=d_{W}\circ q_f - q_f\circ d_{\Cone(f)}\\
&=d_W
\begin{bmatrix}
1&0
\end{bmatrix}
-
\begin{bmatrix}
1&0
\end{bmatrix}
\begin{bmatrix}
d_W & f\\
&-d_V
\end{bmatrix}
=
\begin{bmatrix}
d_W&0
\end{bmatrix}
-
\begin{bmatrix}
d_W&f
\end{bmatrix}
=
\begin{bmatrix}
0& -f
\end{bmatrix}\\
&=-f\circ p_f.
\end{align*}
Hence
\[
\begin{tikzcd}
W \arrow[r,yshift=0.7ex,"{i_f}"] & \Cone(f) \arrow[r,yshift=0.7ex,"{p_{f}}"] \arrow[l,yshift=-0.7ex,"{q_{f}}"] & V[1] \arrow[l,yshift=-0.7ex,"{j_{f}}"]
\end{tikzcd}
\]
is a cone diagram of $f$.
\end{proof}

\begin{cor}\label{cor:exact DG periodic}
Let $\catC$ be an additive category.
\begin{enua}
\item $C_m(\catC)_{\dg}$ is an exact DG category.
\item $C_m(\catC)$ is a Frobenius category.
\item $K_m(\catC)$ is a triangulated category.\qed
\end{enua}
\end{cor}

We describe conflations on a Frobenius exact category $C_m(\catC)$.
A sequence $\epsilon : U\xr{f} V\xr{g} W$ in $C_m(\catC)$
is a conflation if and only if $\epsilon$ is a graded split sequence.
Since a morphism of $C_m(\catC)_{\dg}^0$ is a degree-preserving map,
$\epsilon$ is a graded split 
if and only if $U^i\xr{f^i} V^i\xr{g^i} W^i$ is a split exact sequence
for all $i\in\bbZ_m$.
Next, we describe the projective objects in $C_m(\catC)$.
\begin{dfn}
Let $\catC$ be an additive category.
For each object $A\in\catC$,
define an $m$-periodic complex $K_A \in C_m(\catC)$ as follows.
\[
K_A^i:=
\begin{cases}
A& (i=m-1, 0) \\
0& (i\ne m-1, 0)
\end{cases},\quad
d_{K_A}^i:=
\begin{cases}
\id_A& (i=m-1) \\
0& (i\ne m-1)
\end{cases}
\]
for $m\ge 2$, and
\[
K_A^0:= A\oplus A,\quad
d_{K_A}^0:=
\begin{bmatrix}
0& \id_A\\
0& 0
\end{bmatrix}
\]
for $m=1$. 
\end{dfn}

\begin{lem}\label{lem:contractible periodic}
Let $\catC$ be an additive category, and $A\in\catC$.
\begin{enua}
\item
$K_A$ is a contractible object.
\item
$K_A$ is a projective object of $C_m(\catC)$ 
as an exact category with the graded split exact structure.
\item
$K_A\simeq 0$ in $K_m(\catC)$.
\end{enua}
\end{lem}
\begin{proof}
(2) and (3) follows from (1).
We prove (1).
For $m=1$,
consider the following diagram.
\[
\begin{tikzcd}
A\oplus A\arrow[r,"e_{12}"] \arrow[d,equal] & A\oplus A \arrow[ld,"e_{21}"'] \arrow[r,"e_{12}"] \arrow[d,equal] & A\oplus A \arrow[d,equal] \arrow[ld,"e_{21}"']\\
A\oplus A\arrow[r,"e_{12}"] & A\oplus A \arrow[r,"e_{12}"] & A\oplus A,
\end{tikzcd}
\]
where 
$e_{12}=
\begin{bmatrix}
0&1\\
0&0
\end{bmatrix}$,
and
$e_{21}=
\begin{bmatrix}
0&0\\
1&0
\end{bmatrix}$.
Then $\id_{A\oplus A}=e_{12}e_{21}+e_{21}e_{12}=d\left(e_{21}\right)$.
Thus $\id_{K_A}$ is an exact morphism.
For $m\ge 2$,
consider the following diagram.
\[
\begin{tikzcd}
\cdots \arrow[r] & 0 \arrow[r] \arrow[d] & A \arrow[r,equal] \arrow[d,equal] \arrow[ld] & A\arrow[r] \arrow[d,equal] \arrow[ld,equal] & 0 \arrow[r] \arrow[d] \arrow[ld] & \cdots\\
\cdots \arrow[r] & 0 \arrow[r] & A \arrow[r,equal] & A \arrow[r] & 0 \arrow[r] & \cdots.
\end{tikzcd}
\]
This also shows that $\id_{K_A}$ is an exact morphism.
\end{proof}

We reveal the relationship between periodic complexes and usual complexes.
There are two important DG functors 
which relate periodic complexes with usual complexes.
$C(\catC)_{\dg}$ (resp. $C^b(\catC)_{\dg}$) denotes the DG category of
usual (resp. bounded) complexes over $\catC$.
\begin{dfn}
Let $\catC$ be an additive category.
\begin{enua}
\item
A DG functor $\iota : C_m(\catC)_{\dg} \to C(\catC)_{\dg}$
is defined as follows.
\begin{itemize}[nosep]
\item 
For any $V\in C_m(\catC)_{\dg}$ and $i\in\bbZ$,
define
\[
\iota(V)^i:=V^{\ol{i}},\quad
d_{\iota(V)}^i:=d_V^{\ol{i}}
\]
\item
For any $f\in C_m(\catC)_{\dg}(V, W)^p$ and $i\in\bbZ$,
define
\[
\iota(f)^i:=
\left(f^{\ol{i}}: V^{\ol{i}}\to W^{\ol{i+p}}\right) 
\in C(\catC)_{\dg}(\iota(V),\iota(W))^p.
\]
\end{itemize}
\item
A DG functor $\pi : C^b(\catC)_{\dg} \to C_m(\catC)_{\dg}$
is defined as follows.
\begin{itemize}[nosep]
\item 
For any $V\in C^b(\catC)_{\dg}$ and $i\in\bbZ_m$,
define
\[
\pi(V)^i:=\bigoplus_{\ol{j}=i}V^{j},\quad
d_{\pi(V)}^i:= \bigoplus_{\ol{j}=i} d_V^{j}
\]
\item
For any $f\in C^b(\catC)_{\dg}(V, W)^p$ and $i\in\bbZ_m$,
define
\[
\pi(f)^i:=
\left(\bigoplus_{\ol{j}=i} f^{j}:
\bigoplus_{\ol{j}=i}V^j \to \bigoplus_{\ol{j}=i}W^{j+p}\right)
\in C_m(\catC)_{\dg}(\pi(V),\pi(W))^p.
\]
\end{itemize}
\end{enua}
\end{dfn}

By abuse of notation,
we also denote $Z^0(\iota), H^0(\iota)$ (resp. $Z^0(\pi)$, $H^0(\pi)$)
by $\iota$ (resp. $\pi$).

\begin{rmk}
$\iota : C_m(\catC)_{\dg}\to C(\catC)_{\dg}$ is a faithful DG functor.
Hence $\iota :C_m(\catC) \to C(\catC)$ is also a faithful functor,
but $\iota :K_m(\catC) \to K(\catC)$ is not faithful functor in general.
\end{rmk}

\begin{prp}\label{prp:pi Hom}
Let $\catC$ be an additive category,
and $V, W$ bounded complexes.
\begin{enua}
\item
$C_m(\catC)_{\dg}(\pi V,\pi W)
=\bigoplus_{i\in\bbZ}C^b(\catC)_{\dg}(V,W[mi])$ as complexes.
\item
$C_m(\catC)(\pi V,\pi W)
=\bigoplus_{i\in\bbZ}C^b(\catC)(V,W[mi])$.
\item
$K_m(\catC)(\pi V,\pi W)
=\bigoplus_{i\in\bbZ}K^b(\catC)(V,W[mi])$.
\end{enua}
\end{prp}
\begin{proof}
(2) and (3) follow from (1).
We prove (1).
In the following calculation,
we freely use the fact 
that $V^i=0$ for only finitely many $i\in\bbZ$.
We have
\begin{align*}
\bigoplus_{i\in\bbZ}C^b(\catC)_{\dg}(V,W[mi])
=\bigoplus_{i\in\bbZ}\bigoplus_{p\in\bbZ}\prod_{q\in\bbZ}
\catC\left(V^q,W^{q+p+mi}\right)
=\bigoplus_{i\in\bbZ}\bigoplus_{p\in\bbZ}\bigoplus_{q\in\bbZ}
\catC\left(V^q,W^{q+p+mi}\right),
\end{align*}
and
\begin{align*}
C_m(\catC)_{\dg}\left(\pi V,\pi W\right)
&=\bigoplus_{p\in\bbZ}\bigoplus_{0\le j \le m-1}
\catC\left(\pi V^{\ol{j}},\pi W^{\ol{j+p}}\right)
=\bigoplus_{p\in\bbZ}\bigoplus_{0\le j \le m-1}
\catC\left(\bigoplus_{k\in\bbZ} V^{j+mk},\bigoplus_{l\in\bbZ} W^{j+p+ml}\right)\\
&=\bigoplus_{p\in\bbZ}\bigoplus_{0\le j \le m-1}\bigoplus_{k\in\bbZ}\bigoplus_{l\in\bbZ}
\catC\left( V^{j+mk}, W^{j+p+ml}\right)
=\bigoplus_{p\in\bbZ}\bigoplus_{q\in\bbZ}\bigoplus_{i\in\bbZ}
\catC\left( V^{q}, W^{q+p+mi}\right).
\end{align*}
In the last equality, we set $q:=j+mk$ and $i:=l-k$.
Hence we obtain the desired equality.
\end{proof}

Let $\catB$ be an abelian category.
Then $C_m(\catB)$ is also an abelian category.
A sequence
\[
0\to U\xr{f} V\xr{g} W\to 0
\]
in $C(\catB)$ is exact if and only if
\[
0\to U^i\xr{f^i} V^i\xr{g^i} W^i\to 0
\]
is exact in $\catB$ for all $i\in\bbZ_m$.

\begin{prp}\label{prp:exact functor}
Let $\catB$ be an abelian category.
\begin{enua}
\item 
$\iota :C_m(\catB) \to C(\catB)$ is an exact functor of abelian categories.
\item
$\pi : C^b(\catB) \to C_m(\catB)$ is an exact functor of abelian categories.
\end{enua}
\end{prp}
\begin{proof}
(1)
Let $f:V\to W$ be a chain map in $C_m(\catB)$.
Then
we have
\begin{align*}
\Ker(\iota(f))^i
=\Ker(\iota(f)^i)
=\Ker(f^{\ol{i}}),\quad
\Ima(\iota(f))^i
=\Ima(\iota(f)^i)
=\Ima(f^{\ol{i}}),
\end{align*}
for any $i \in \bbZ$.
Hence $\iota$ is an exact functor.

(2)
Let $f:V\to W$ be a chain map in $C^b(\catB)$.
Then
we have
\begin{align*}
\Ker(\pi(f))^i
&=\Ker(\pi(f)^i)
=\Ker\left(\bigoplus_{\ol{j}=i} f^{j}:
\bigoplus_{\ol{j}=i}V^j \to \bigoplus_{\ol{j}=i}W^j\right)
=\bigoplus_{\ol{j}=i}\Ker(f^j:V^j\to W^j),\\
\Ima(\pi(f))^i
&=\bigoplus_{\ol{j}=i}\Ima(f^j:V^j\to W^j)
\end{align*}
for any $i \in \bbZ_m$.
Hence $\pi$ is an exact functor.
\end{proof}

Cohomology theory also work on periodic complexes,
and hence we can define the derived category for periodic complexes.
\begin{dfn}
Let $\catB$ be an abelian category,
$V\in C_m(\catB)$
and $i\in\bbZ_m$.
\begin{enua}
\item
$Z^i(V):=\Ker\left(d_V^i: V^i \to V^{i+1} \right) \in \catB$
is called the $i$th cocycle of $V$.
\item
$B^i(V):=\Ima\left(d_V^i: V^{i-1} \to V^{i} \right) \in \catB$
is called the $i$th coboundary of $V$.
\item
Note that $B^i(V)\subset Z^i(V)$ because $d_V^i\circ d_V^{i-1}=0$.
$H^i(V):=Z^i(V)/B^i(V) \in \catB$
is called the $i$th cohomology of $V$.
\end{enua}
\end{dfn}

\begin{prp}\label{prp:long exact periodic}
Let $\catB$ be an abelian category.
\begin{enua}
\item
$Z^i(\iota (V)) = Z^{\ol{i}}(V)$,
$B^i(\iota (V)) = B^{\ol{i}}(V)$,
and $H^i(\iota (V)) = H^{\ol{i}}(V)$
holds for any $V\in C_m(\catB)$ and $i\in\bbZ$.
In particular,
$Z^{\ol{i}}(-)$, $B^{\ol{i}}(-)$ and $H^{\ol{i}}(-)$ 
define functors $C_m(\catB) \to \catB$,
and $H^{\ol{i}}(-)$ defines a functor $K_m(\catB) \to \catB$.
\item
$H^i(\pi(V))=\bigoplus_{\ol{j}=i}H^j(V)$ for any $V\in C^b(\catB)$ and $i\in \bbZ_m$
\item
$H^i(-) : K_m(\catB) \to \catB$ is a cohomological functor.
\item
Let $0\to U\xr{f} V\xr{g} W\to0$ be an exact sequence in $C_m(\catB)$.
Then there exists a natural morphism $\delta^i :H^i(W)\to H^{i+1}(U)$
such that a sequence
\[
H^{i-1}(W) \xr{\delta^{i-1}} H^i(U) \xr{H^i(f)} H^i(V) \xr{H^i(g)} H^i(W) \xr{\delta^i} H^{i+1}(U)
\]
is exact for all $i\in\bbZ_m$.
\end{enua}
\end{prp}
\begin{proof}
(1) is clear.
(3) follows from (1) and 
the fact that $H^i(-):K(\catB) \to B$ is a cohomological functor for usual complexes.
See Proposition \ref{prp:delta} below for (4).
We will prove (2).
Since we have
\begin{align*}
Z^i(\pi(V))
&=\Ker(d^i_{\pi(V)})
=\Ker\left( \bigoplus_{\ol{j}=i} d_V^{j}:
\bigoplus_{\ol{j}=i}V^j \to \bigoplus_{\ol{j}=i}V^{j+1} \right)\\
&=\bigoplus_{\ol{j}=i}\Ker\left( d^j: V^j \to V^{j+1} \right)
=\bigoplus_{\ol{j}=i}Z^j(V),\\
B^i(\pi(V))
&=\bigoplus_{\ol{j}=i}B^j(V)
\end{align*}
for any $i\in\bbZ_m$.
Hence we obtain
\[
H^i(\pi(V))=\bigoplus_{\ol{j}=i}H^j(V)
\]
\end{proof}

\begin{dfn}\label{dfn:periodic acyclic}
Let $\catB$ be an abelian category.
\begin{enua}
\item
$V\in C_m(\catB)$ is \emph{acyclic}
if $H^i(V)=0$ for all $i\in \bbZ_m$.
\item
A chain map $f:V\to W$ in $C_m(\catB)$ is a \emph{quasi-isomorphism}
if $H^i(f):H^i(V)\to H^i(W)$ is an isomorphism for all $i\in\bbZ_m$.
\item
$K_{m,a}(\catB)$ is the full subcategory of $K_m(\catB)$
consisting of acyclic $m$-periodic complexes.
\end{enua}
\end{dfn}

\begin{dfn}\label{dfn:periodic derived}
The \emph{$m$-periodic derived category} of an abelian category $\catB$
is the Verdier quotient $K_m(\catB)/K_{m,a}(\catB)$.
\end{dfn}

\begin{lem}\label{lem:qis cone periodic}
Let $\catB$ be an abelian category.
\begin{enua}
\item 
$K_{m,a}(\catB)$ is a thick subcategory of $K_m(\catB)$.
\item 
A chain map $f:V\to W$ in $C(\catB)$ is a quasi-isomorphism
if and only if
$\Cone(f)$ is acyclic.
\item 
If $V\in C^b(\catB)$ (resp. $C_m(\catB)$) is acyclic,
then $\pi(V)\in C_m(\catB)$ (resp. $\iota(V)\in C(\catB)$) is also acyclic.
\item 
If $f:V\to W$ in $C^b(\catB)$ (resp.\ $C_m(\catB)$) is a quasi-isomorphism,
then $\pi(f)$ (resp.\ $\iota(f)$) is also a quasi-isomorphism.
\end{enua}
\end{lem}
\begin{proof}
It follows from Proposition \ref{prp:long exact periodic}.
\end{proof}

\begin{cor}
The $m$-periodic derived category $D_m(\catB)$ of an abelian category $\catB$
is the localization of $K_m(\catB)$ with respect to quasi-isomorphisms.
\qed
\end{cor}

By the universal property of Verdier quotients,
$\pi$ and $\iota$ induce triangulated functors
$D^b(\catB) \to D_m(\catB)$ and $D_m(\catB)\to D(\catB)$, respectively.
We also denote it by $\pi$ (resp. $\iota$).

\begin{dfn}
Let $\catB$ be an abelian category.
\begin{enua}
\item
$V\in C_m(\catB)$ is \emph{$K$-projective}
if $K_m(\catB)(V,W)=0$ for any acyclic $W\in C_m(\catB)$.
\item
$K_{m,p}(\catB)$ denotes the full subcategory of $K_m(\catB)$
consisting of $K$-projective $m$-periodic complexes.
\end{enua}
\end{dfn}

Note that $K_{m,p}(\catB)$ is a thick subcategory of $K_m(\catB)$.

\begin{prp}
Let $\catB$ be an abelian category.
The following are equivalent for $V\in C_m(\catB)$.
\begin{enua}
\item
$V$ is $K$-projective.
\item
The natural morphism $K_m(\catB)(V,W) \to D_m(\catB)(V,W)$
is an isomorphism for all $W\in C_m(\catB)$.
\end{enua}
\end{prp}
\begin{proof}
$(1)\Rightarrow (2)$:
Note that any quasi-isomorphism $s:W\to V$ is a retraction.
Indeed,
considering an exact triangle
\[
W\xr{s} V \xr{f} \Cone(s) \xr{g} X[1],
\]
we have $f=0$ since $\Cone(s)$ is acyclic.
Hence $s$ is a retraction.
We will show that $\pi$ is injective.
If $f\in K_m(\catB)(V,W)$ with $\pi(f)=0$,
there exists a morphism $s:V'\to V$ such that $fs=0$.
Since $s$ is a retraction, $f=0$.
We will show that $\pi$ is surjective.
For any morphism $f:V\to W$ in $D_m(\catB)$,
there exists a morphism $g:Z\to W$ in $K_m(\catB)$
and a quasi-isomorphism $s:Z\to V$,
$f=\pi(g)\pi(s)^{-1}$.
Since $s$ is a retraction, there exists a morphism $u:V\to Z$ such that $su=\id_T$.
Hence $f=\pi(g)\pi(s)^{-1}=\pi(g)\pi(u)=\pi(gu)$.

$(2)\Rightarrow (1)$:
For any acyclic $m$-periodic complex $W$, we have
\[
K_m(\catB)(V, W)
\simeq D_m(\catB)(V,W)
\simeq D_m(V,0)
=0.
\]
\end{proof}

As in the case of the usual derived category,
the natural functor $\catB \to D_m(\catB)$
is a $\delta$-functor.
A \emph{$\delta$-functor} $\catB\to\catT$
from an abelian category $\catB$ to a triangulated category $\catT$
is a pair $(F,\delta)$
of a functor $F:\catA\to\catT$
and functorial morphisms $\delta_{W,U}:\Ext^1_{\catA}(W,U)\to \catT(FW,\Sg FU)$
for $U,W\in\catB$
such that
for any exact sequence 
$\epsilon : 0\to U \xr{f}V\xr{g}W\to0$ in $\catB$,
\[
FU\xr{Ff}FV\xr{Fg}FW\xr{\delta(\epsilon)}\Sg FU
\]
is an exact triangle in $\catT$.
If no ambiguity can arise,
we will often say that $F:\catB \to \catT$ is a $\delta$-functor.

\begin{prp}\label{prp:delta}
Let $\catB$ be an abelian category.
Then the natural functor $C_m(\catB) \to D_m(\catB)$ is a $\delta$-functor.
\end{prp}
\begin{proof}
Consider the following commutative diagram in $C_m(\catB)$.
\[
\begin{tikzcd}
&0 \arrow[d] &0 \arrow[d] &&\\
0 \arrow[r] & U \po \arrow[r] \arrow[d,"f"'] & \Cone(\id_U) \arrow[r] \arrow[d] &U[1] \arrow[r] \arrow[d,equal] &0\\
0 \arrow[r] & V \arrow[r,"i_f"'] \arrow[d,"g"'] & \Cone(f) \arrow[r,"p_f"'] \arrow[d,"\phi"] & U[1] \arrow[r] & 0\\
&W \arrow[r,equal] \arrow[d] &W \arrow[d] \\
&0&0
\end{tikzcd},
\]
where vertical and horizontal sequences are exact.
Since $\Cone(\id_U)$ is acyclic,
$\phi : \Cone(f) \to W$ is a quasi-isomorphism.
Hence in $D_m(\catB)$
\[
U\xr{f} V\xr{g} W \xr{p_f\phi^{-1}} U[1]
\]
is isomorphic to an exact triangle
\[
U\xr{f} V\xr{i_f} \Cone(f) \xr{p_f} U[1].
\]
\end{proof}

\subsection{Periodic complexes over an abelian category of finite global dimension}\label{ss:periodic fin global}
In this subsection, 
$\catB$ is an abelian category of finite global dimension $d$,
which has enough projective $\catP$.
We will prove that
$K(\catB)$ has ``enough $K$-projective'' in this case (Proposition \ref{prp:K-proj wtP} and \ref{prp:enough K-proj periodic} ).
Consequently, we prove the natural functor $K_m(\catP)\to D_m(\catB)$ is an equivalence
(Corollary \ref{cor:proj periodic}).

We first consider $K_m(\catP)$ and 
observe that it behaves like the bounded homotopy category $K^b(\catP)$.
We give a characterization of projective objects in $C_m(\catP)$.
\begin{lem}[c.f. {\cite[Lemma 3.2.]{Br13}}]\label{lem:proj acy}
The following are equivalent for $V\in C_m(\catP)$.
\begin{enua}
\item
$V$ is a contractible object.
\item
$V$ is a projective object of $C_m(\catP)$.
\item
$V\simeq 0$ in $K_m(\catP)$.
\item
$V\in C_m(\catB)$ is acyclic.
\item
There exist $P_1,\dots, P_n \in \catP$
and $\ell_1, \dots, \ell_n \in \bbZ$
such that
\[
V \simeq \bigoplus_{1\le i \le n} K_{P_i}[\ell_i]
\]
in $C_m(\catB)$.
\end{enua}
\end{lem}
\begin{proof}
$(1)\Leftrightarrow (2) \Leftrightarrow (3)$:
See \S \ref{ss:dg}.

$(3)\Rightarrow (4)$:
Since cohomologies are invariant under homotopy equivalence,
$H^i(V)\simeq H^i(0)=0$.

$(4)\Rightarrow (5)$:
For all $i\in \bbZ_m$,
there exists an exact sequence
\begin{align}\label{eq:ses acy}
0\to Z^i(V) \to V^i \to Z^{i+1}(V) \to 0
\end{align}
in $\catB$.
Considering the long exact sequences of $\Ext^j_{\catB}(-,W)$,
we have
\[
\Ext^j_{\catB}\left(Z^i(V),-\right)
\simeq \Ext^{j+1}_{\catB}\left(Z^{i+1}(V),-\right)
\simeq \Ext^{j+m}_{\catB}\left(Z^{i+m}(V),-\right)
= \Ext^{j+m}_{\catB}\left(Z^{i}(V),-\right)
\]
for $j>0$.
Since $\catB$ has finite global dimension,
$\Ext^j_{\catB}\left(Z^i(V),-\right)=0$ for $j\gg 0$.
Thus $Z^i(V)$ is projective for all $i\in \bbZ_m$.
Hence the exact sequence \eqref{eq:ses acy} splits,
and we obtain 
\[
V \simeq \bigoplus_{1\le i \le m} K_{Z^i(V)}[-i]
\]
in $C_m(\catB)$.

$(5)\Rightarrow (1)$
It follows from Lemma \ref{lem:contractible periodic}.
\end{proof}

\begin{prp}\label{prp:qis proj}
The following are equivalent for a chain map $f:V\to W$
in $C_m(\catP)\subset C_m(\catB)$.
\begin{enua}
\item $f$ is a quasi-isomorphism.
\item $f$ is an isomorphism in $K_m(\catB)$.
\item $\Cone(f)$ is acyclic.
\end{enua}
\end{prp}
\begin{proof}
$(1)\Leftrightarrow (3)$:
It is Lemma \ref{lem:qis cone periodic}.

$(2)\Rightarrow (1)$:
It is clear.

$(3)\Rightarrow (2)$:
Suppose $\Cone(f)$ is acyclic.
Note that $\Cone(f)\in C_m(\catP)$.
Then $\Cone(f)=0$ in $K_m(\catB)$ by Lemma \ref{lem:proj acy},
and hence $f:V\to W$ is an isomorphism in $K_m(\catB)$.
\end{proof}

\begin{prp}\label{prp:ex seq tri proj}
Let $0\to U\xr{f} V\xr{g} W\to 0$ be an exact sequence in $C_m(\catB)$
with $U,W\in C_m(\catP)$.
Then there exists $h:W \to U[1]$ such that a triangle
\begin{align}\label{eq:ex tri 1}
U\xr{f} V\xr{g} W \xr{h} U[1]
\end{align}
is exact in $K_m(\catB)$.
\end{prp}
\begin{proof}
Consider the following commutative diagram in $C_m(\catB)$.
\[
\begin{tikzcd}
&0 \arrow[d] &0 \arrow[d] &&\\
0 \arrow[r] & U \po \arrow[r] \arrow[d,"f"'] & \Cone(\id_U) \arrow[r] \arrow[d] &U[1] \arrow[r] \arrow[d,equal] &0\\
0 \arrow[r] & V \arrow[r,"i_f"'] \arrow[d,"g"'] & \Cone(f) \arrow[r,"p_f"'] \arrow[d,"\phi"] & U[1] \arrow[r] & 0\\
&W \arrow[r,equal] \arrow[d] &W \arrow[d] \\
&0&0
\end{tikzcd},
\]
where vertical and horizontal sequences are exact.
Note that all objects in this diagram are in $C_m(\catP)$.
Since $\Cone(\id_U)$ is acyclic,
$\phi : \Cone(f) \to W$ is a quasi-isomorphism,
and hence it is an isomorphism in $K_m(\catB)$ by Proposition \ref{prp:qis proj}.
Thus
\[
U\xr{f} V\xr{g} W \xr{p_f\phi^{-1}} U[1]
\]
is isomorphic to an exact triangle
\[
U\xr{f} V\xr{i_f} \Cone(f) \xr{p_f} U[1].
\]
\end{proof}

Let $\catI$ (resp. $\catJ$) be the class of stalk complexes 
in $C_m(\catB)$ (resp. $C_m(\catP)$).
$\wt{\catB}$ (resp. $\wt{\catP}$) denotes
the smallest extension closed subcategory of an abelian category $C_m(\catB)$ 
containing of $\catI$ (resp. $\catJ$).
Note that $\wt{\catP} \subset C_m(\catP)$

\begin{prp}[c.f. {\cite[Lemma 9.4.]{Go13}}]\label{prp:K-proj wtP}
$\wt{\catP}\subset K_{m,p}(\catB)$ holds.
\end{prp}
\begin{proof}
We first prove that every $V\in \catJ$ is $K$-projective.
We may assume that $V$ is concentrated in degree $0$, and $V^0=P\in\catP$.
For any chain map $f : P \to W$ in $C_m(\catB)$ with $W$ acyclic,
consider the following diagram.
\[
\begin{tikzcd}
\cdots \arrow[r] & 0 \arrow[r] \arrow[d] & P \arrow[r] \arrow[d,"f^0"] \arrow[ld,"g"'] & 0 \arrow[r] \arrow[d] \arrow[ld] & \cdots \\
\cdots \arrow[r] & W^{m-1} \arrow[r,"d_W^{m-1}"'] & W^0 \arrow[r,"d_W^{0}"'] & W^1 \arrow[r] & \cdots ,
\end{tikzcd}
\]
where $g$ is a lift of $f^0$ induced by the projectivity of $P$.
The above diagram shows that $f$ is an exact morphism,
and hence $f=0$ in $K_m(\catB)$. 
Hence we have $\catJ \subset K_{m,p}(\catB)$.
Thus we obtain $\wt{\catP}\subset \thick_{K_m(\catB)}(\catJ)\subset K_{m,p}(\catB)$,
where the first inclusion follows from Proposition \ref{prp:ex seq tri proj}.
\end{proof}

\begin{prp}[c.f. {\cite[Proposition 3.3.]{Go13}}]\label{prp:Ext K-proj}
$\Ext^i_{C_m(\catB)}(P,V) \simto K_m(\catB)(P[-i],V)$
for any $P\in C_m(\catP)$, $V\in C_m(\catB)$ and $i\ge1$.
\end{prp}
\begin{proof}
We first prove that $\Cone(\id_P)$ is 
a projective object of $C_m(\catB)$ as an abelian category.
Let $\epsilon : 0 \to U\to V\to \Cone(\id_P)\to 0$ be an exact sequence in $C_m(\catB)$.
Then $\epsilon$ is a graded split exact sequence
since $\Cone(\id_P)\in C_m(\catP)$.
Because $\Cone(\id_P)$ is a projective object of $C_m(\catB)$ 
as an exact category with the graded split exact structure,
$\epsilon $ splits.
It shows $\Cone(\id_P)$ is a projective object of $C_m(\catB)$ as an abelian category.

Set $R_i:=\Cone(\id_{P[-i-1]})$ .
Then we have the canonical exact sequence
\[
0\to P[-i-1] \to R_i \to P[-i]\to 0
\]
of $C_m(\catB)$.
Thus we have a complex of complexes
\[
R_{\bu}:=\left( \cdots \to R_2 \to R_1 \to R_0 \to 0 \to \cdots  \right)
\]
with a quasi-isomorphism $R_{\bu}\to P$ in $C\left( C_m(\catP) \right)$.
Since $R_i$ is a projective object in $C_m(\catB)$,
$R_{\bu}$ is a projective resolution of $P$.
Thus we have
\[
\Ext^p_{C_m(\catB)}(P,V)=H^i\left(C_m(\catB)(R_{\bu}, V) \right)
\]
The cocycle $Z^i\left(C_m(\catB)(R_{\bu}, V) \right)$
coincides with $C_m(\catB)(P[-i],V)$,
and the coboundary $B^i\left(C_m(\catB)(R_{\bu}, V) \right)$ 
coincides with the set of morphisms
which factor through $R_{i-1}=\Cone(\id_{P[-i-2]})$. 
Hence we obtain
\[
H^i\left(C_m(\catB)(R_{\bu}, V) \right)
=K_m(\catB)(P[-i],V).
\]
\end{proof}

\begin{lem}[{\cite[Lemma 9.3.]{Go13}}]\label{lem:proj resol}
$\pi\left(C^b(\catP)\right) \subset \wt{\catP}$ holds.
\end{lem}
\begin{proof}
Let $P\in C^b(\catP)$.
We may assume that $P^i=0$ for $i<-n, 0<i$, where $n\in \bbZ_{\ge0}$. 
Then we have the following exact sequences in $C^b(\catB)$.
\[
\begin{tikzcd}[row sep=tiny]
0 \arrow[r] & P^{-n}[n] \arrow[r] & P \arrow[r] & \sigma_{\ge -n+1}P \arrow[r] & 0\\
0 \arrow[r] & P^{-n+1}[n-1] \arrow[r] & \sigma_{\ge -n+1}P \arrow[r] & \sigma_{\ge -n+2}P \arrow[r] & 0\\
0 \arrow[r] & P^{-n+2}[n-2] \arrow[r] & \sigma_{\ge -n+2}P \arrow[r] & \sigma_{\ge -n+3}P \arrow[r] & 0\\
&& \cdots\\
0 \arrow[r] & P^{-1}[1] \arrow[r] & \sigma_{\ge -1}P \arrow[r] & P^0 \arrow[r] & 0 ,\\
\end{tikzcd}
\]
where
\[
\sigma_{\ge i}P
:=\left(\cdots \to 0  \to 0 \to P^i \to P^{i+1} \to P^{i+2} \to \cdots  \right).
\]
Since $\pi : C^b(\catB)\to C_m(\catB)$ is exact,
we have the desired conclusion.
\end{proof}

\begin{lem}[c.f. {\cite[Proposition 9.7.]{Go13}}]\label{lem:ext stalk}
$\wt{\catB}= C_m(\catB)$.
\end{lem}
\begin{proof}
$\wt{\catB}\subset C_m(\catB)$ is clear.
Let $V\in C_m(\catB)$,
and $n_V$ be the number of $i\in \bbZ_m$ with $V^i \ne 0$.
We prove $V\in\wt{\catB}$ by induction on $n_V$.
If $n_V=0$, $V=0$.
Hence $V\in C_m(\catB)$.
If $n_V\ge 1$, there exists $i\in \bbZ_m$ such that $V^i\ne0$.
We may assume that $i=0$.
Then there exists the following exact sequences in $C_m(\catB)$.
\[
\begin{tikzcd}[row sep=scriptsize,column sep=small]
0 \arrow[d] && 0 \arrow[d] & 0\arrow[d] & 0\arrow[d] \\
U \arrow[d] & : (\cdots \arrow[r] & V^{m-1} \arrow[r] \arrow[d,equal] & Z^0(V) \arrow[r,"0"] \arrow[d] & V^1 \arrow[r] \arrow[d,equal] & \cdots )\\
V \arrow[d] & : (\cdots \arrow[r] & V^{m-1} \arrow[d] \arrow[r] & V^0 \arrow[r] \arrow[d] & V^1 \arrow[r] \arrow[d] & \cdots ) \\
V^0/Z^0(V) \arrow[d] & : (\cdots \arrow[r] & 0 \arrow[d] \arrow[r] & V^0/Z^0(V) \arrow[d] \arrow[r] & 0 \arrow[r] \arrow[d] & \cdots ), \\
0&& 0 & 0 & 0
\end{tikzcd}
\]
and
\[
\begin{tikzcd}[row sep=scriptsize,column sep=small]
0 \arrow[d] && 0 \arrow[d] & 0\arrow[d] & 0\arrow[d] \\
Z^0(V) \arrow[d] & : (\cdots \arrow[r] & 0 \arrow[r] \arrow[d] & Z^0(V) \arrow[r] \arrow[d] & 0 \arrow[r] \arrow[d] & \cdots )\\
U \arrow[d] & : (\cdots \arrow[r] & V^{m-1} \arrow[r] \arrow[d,equal] & Z^0(V) \arrow[r,"0"] \arrow[d] & V^1 \arrow[r] \arrow[d,equal] & \cdots )\\
W \arrow[d] & : (\cdots \arrow[r] & V^{m-1} \arrow[d] \arrow[r] & 0 \arrow[d] \arrow[r] & V^1 \arrow[r] \arrow[d] & \cdots ) .\\
0&& 0 & 0 & 0
\end{tikzcd}
\]
Since $n_W < n_V$,
we have $W\in \wt{\catB}$ by the hypothesis of induction.
Thus $V \in \wt{\catB}$.
\end{proof}

The following proposition says that $K_m(\catB)$ has ``enough $K$-projectives''.
\begin{prp}[{\cite[Lemma 9.5.]{Go13}}]\label{prp:enough K-proj periodic}
For any $V\in C_m(\catB)$,
there is a surjective quasi-isomorphism $f:P \surj V$ with $P\in\wt{\catP}$.
\end{prp}
\begin{proof}
Let $\catI$ be the classes of stalk complexes in $C_m(\catB)$.
Consider a full subcategory
\[
\catU :=\left\{V\in C_m(\catB) \mid 
\mbox{There exists a surjective quasi-isomorphism $f_M:P_V \surj V$ with $P_V\in\wt{\catP}$}
\right\}.
\]
We first show that $\catI \subset \catU$,
and next $\catU$ is an extension closed subcategory.
If we prove this, 
then $C_m(\catB) = \wt{\catB}  \subset \catU $
by Lemma \ref{lem:ext stalk},
and the proof is complete.

Let $V\in\catI$.
We may assume that $V$ is concentrated in degree $0$.
Let
\[
0\to P^{-d}\to P^{-d+1} \to \dots \to P^0 \to V \to 0
\]
be a finite projective resolution of $V$ in $\catB$.
Set
\[
P:=\left(\cdots\to 0\to P^{-d}\to P^{-d+1} \to \dots \to P^0 \to 0 \to \cdots \right)
\in C^b(\catP),
\]
then we have a surjective quasi-isomorphism $f:P\to V$ in $C^b(\catB)$.
By Proposition \ref{prp:exact functor} (2) and Lemma \ref{lem:qis cone periodic} (4),
$\pi(f):\pi(P)\surj V$ is also a surjective quasi-isomorphism.
Since $\pi(P)\in\wt{\catP}$ by Lemma \ref{lem:proj resol},
$\pi(f)$ is the desired morphism.
Hence $V\in \catU$.

Next, we show that $\catU$ is closed under extension.
Let $\epsilon : 0\to U \to V\to W\to 0$ be an exact sequence
of $C_m(\catB)$ with $U,V \in \catU$.
There exist surjective quasi-isomorphisms $P_U\surj U$ and $P_W\surj W$
with $P_U, P_W \in \wt{\catP}$.
Pulling-back $\epsilon$ by $P_W\surj W$,
We get the following commutative diagram in $C_m(\catB)$:
\[
\begin{tikzcd}
\epsilon': &
 0 \arrow[r] & U \arrow[r] \arrow[d,equal] & Q \pb \arrow[r] \arrow[d,two heads] & P_W \arrow[r] \arrow[d,two heads] & 0\\
\epsilon : &
 0 \arrow[r] & U \arrow[r] & V \arrow[r] & W \arrow[r] & 0.
\end{tikzcd}
\]
Let $A_U$ be the kernel of $P_U \surj U$.
Since $P_U \surj U$ is a quasi-isomorphism,
$A_U$ is acyclic.
Applying the functor $C_m(\catB)(P_W, - )$ to
the exact sequence $0\to A_U \to P_U \to U \to 0$
and considering the long exact sequence,
we have
\[
\Ext^1_{C_m(\catB)}(P_W, P_U) \simto \Ext^1_{C_m(\catB)}(P_W, U)
\]
since 
\[
\Ext^p_{C_m(\catB)}(P_W,A_U) \simeq D_m(\catB)(P_W,A_U[p])=0
\]
for all $p\ge1$ by Proposition \ref{prp:Ext K-proj}.
Thus there exists an exact sequence $\epsilon'':0\to P_U\to P_V \to P_W \to 0$
such that the push-out of $\epsilon''$ by $P_U\surj U$ is $\epsilon'$.
By discussion above, we have the following commutative diagram in $C_m(\catB)$.
\[
\begin{tikzcd}
\epsilon'': &
0 \ar[r] & P_U \po \arrow[r] \arrow[d,two heads] & P_V \arrow[r] \arrow[d,two heads] & P_W \arrow[r] \arrow[d,equal] & 0\\
\epsilon': &
 0 \arrow[r] & U \arrow[r] \arrow[d,equal] & Q \pb \arrow[r] \arrow[d,two heads] & P_W \arrow[r] \arrow[d,two heads] & 0\\
\epsilon : &
 0 \arrow[r] & U \arrow[r] & V \arrow[r] & W \arrow[r] & 0.
\end{tikzcd}
\]
By the snake lemma,
$P_V \to V$ is a surjective quasi-isomorphism.
Since $\wt{\catP}$ is closed under extensions in $C_m(\catB)$,
$P_V\in \wt{\catP}$.
Therefore $\catU$ is closed under extensions.
\end{proof}

Recall that $\catJ$ is the class of stalk complexes in $C_m(\catP)$
(See the sentence preceding Proposition \ref{prp:K-proj wtP}).
\begin{cor}\label{cor:proj periodic}
The following holds.
\begin{enua}
\item
$\tria_{K_m(\catB)}(\catP)=K_m(\catP)$.
\item
$K_m(\catP)\subset K_{m,p}(\catB)$ holds.
\item
The natural functor
$K_m(\catP) \to D_m(\catB)$
is a triangulated equivalence.
\item
$\Ext^i_{C_m(\catB)}(P,V) \simto D_m(\catB)(P[-i],V)$
for any $P\in C_m(\catP)$, $V\in C_m(\catB)$ and $i\ge1$.
\end{enua}
\end{cor}
\begin{proof}
For any $V\in C_m(\catP)$,
there exists a surjective quasi-isomorphism $f: P \surj V$ with $P\in\wt{\catP}$.
Then $f$ is an isomorphism in $K_m(\catB)$ by Proposition \ref{prp:qis proj}.
By Proposition \ref{prp:ex seq tri proj}, $P\in\tria_{K_m(\catB)}(\catP)$,
and hence (1) holds.
(2), (3) and (4) follows from (1).
\end{proof}

\begin{cor}\label{cor:Hom periodic}
The following holds.
\begin{enua}
\item
For any $V,W \in C^b(\catB)$,
$D_m(\catB)(\pi V,\pi W)=\bigoplus_{i\in\bbZ}D^b(\catB)(V,W[mi])$.
\item
For any $M,N \in \catB$,
$D_m(\catB)(M,N)=\bigoplus_{i\in\bbZ}\Ext_{\catB}^{mi}(M,N)$.
\item
If $\gd \catB < m$,
then there exists a fully faithful functor
$\catB \inj D_m(\catB)$.
\end{enua}
\end{cor}
\begin{proof}
(1)
Let $f:P\to V$ and $g:Q\to W$ be quasi-isomorphisms with $P,Q \in C^b(\catP)$.
Then $\pi(f):\pi P \to \pi V$ and $\pi(g): \pi Q \to \pi W$
are also quasi-isomorphism with $\pi P, \pi Q\in C_m(\catP)$
by Lemma \ref{lem:qis cone periodic}.
Hence we have
\begin{align*}
D_m(\catB)(\pi V,\pi W)
&\simeq D_m(\catB)(\pi P,\pi Q)
\simeq K_m(\catB)(\pi P,\pi Q)
\simeq \bigoplus_{i\in\bbZ}K^b(\catB)(P,Q[mi])\\
&\simeq \bigoplus_{i\in\bbZ}D^b(\catB)(P,Q[mi])
\simeq \bigoplus_{i\in\bbZ}D^b(\catB)(V,W[mi]),
\end{align*}
where we use Proposition \ref{prp:pi Hom} in the third isomorphism,.

(2)
It follows from (1).

(3)
A natural functor $\catB \to D^b(\catB) \xr{\iota} D_m(\catB)$ is fully faithful.
Indeed,
for any $M,N\in \catB$,
\[
D_m(\catB)(M,N)=\bigoplus_{i\in\bbZ}\Ext_{\catB}^{mi}(M,N)=\catB(M,N)
\]
since $\gd \catB <m$.
\end{proof}

There is a natural question
\footnote{This question is posed by Haruhisa Enomoto.
See \cite[Remark 3.20]{AET21} for the application.}
that when is $\catB$ an extension-closed subcategory of $D_m(\catB)$?
We answer this question.
A subcategory $\catC$ of a triangulated category is \emph{extension-closed}
if for any exact triangle $X\to Y \to Z \to X[1]$ with $X, Z \in \catC$,
$Y$ is also in $\catC$.
\begin{lem}\label{lem:delta ext}
Let $F:\catB \to \catT$ be a $\delta$-functor.
If $\delta_{W,U}$ is surjective for any $U,W\in\catB$,
then the essential image $F(\catB)$ of $F$ 
is an extension-closed subcategory of $\catT$.
\end{lem}
\begin{proof}
Let $U,W\in\catB$,
and $FU \to X \to FW \xr{\delta} \Sg FU$ be an exact triangle in $\catT$.
Since $\delta_{C,A}$ is surjective,
there exists an exact sequence
$\epsilon : 0\to U \xr{f} V \xr{g}W\to0$ in $\catB$ satisfying the following.
\begin{itemize}[nosep]
\item $\delta(\epsilon)=\delta$.
\item $FU \xr{Ff} FV \xr{Fg}FW\xr{\delta}\Sg FU$ is an exact triangle in $\catT$.
\end{itemize}
Then there exists a morphism $\phi : X\to FV$
such that
\[
\begin{tikzcd}
FU \arrow[r] \arrow[d,equal] & X \arrow[r] \arrow[d,"\phi"] & FW \arrow[r,"\delta"] \arrow[d,equal] & \Sg FA \arrow[d,equal]\\
FU \arrow[r,"Ff"] & FV \arrow[r,"Fg"] & FW \arrow[r,"\delta"] & \Sg FU
\end{tikzcd}
\]
commutes,
and hence $\phi$ is an isomorphism.
Thus $X\in F(\catB)$
\end{proof}
\begin{cor}
If $2\le \gd \catB\le m$,
the essential image of the natural functor $\catB \to D_m(\catB)$
is extension-closed.
In particular,
$\catB$ is an extension-closed subcategory of $D_m(\catB)$
if $2\le \gd \catB <m$.
\end{cor}
\begin{proof}
For any $U,W\in \catB$,
the natural morphism $\delta_{W,U}:\Ext^1_\catB(W,U)\to D_m(\catB)(W,U[1])$
is an isomorphism.
Indeed,
\[
D_m(\catB)(W,U[1])
=\bigoplus_{i\in\bbZ}D^b(\catB)(W,U[mi+1])
=\bigoplus_{i\in\bbZ}\Ext_{\catB}^{mi+1}(W,U)
=\bigoplus_{i\ge 0}\Ext_{\catB}^{mi+1}(W,U)
=\Ext_{\catB}^{1}(W,U).
\]
In the third equality, we use $m\ge2$.
In the last equality, we use $\gd \catA \le m$.
Hence the essential image of the natural functor $\catB \to D_m(\catB)$
is extension-closed by Lemma \ref{lem:delta ext}
\end{proof}

In the hereditary case,
the periodic derived category $D_m(\catB)$ behaves similarly to
the bounded derived category $D^b(\catB)$,
that is,
periodic complexes are determined by its cohomology group,
up to quasi-isomorphisms.
\begin{prp}[{\cite[Lemma 4.2.]{Br13}}]\label{prp:hered periodic}
Let $\catB$ be hereditary, i.e. $\gd\catB \le 1$.
Then for any $V\in C_m(\catB)$,
$V\simeq \bigoplus_{0\le i \le m-1} H^i(V)[-i]$ in $D_m(\catB)$.
\end{prp}
\begin{proof}
See \cite[Lemma 5.1]{St18} for the case $m=1$.
Suppose $m\ge 2$.
We may assume that $V\in K_m(\catP)$.
Then we have exact sequence
\[
\epsilon_i : 0 \to Z^i(V) \to V^i \to B^{i+1}(V) \to 0
\]
in $\catB$ for any $i\in\bbZ_m$.
Since $\catB$ is hereditary,
$B^{i+1}(V) \subset V^{i+1}$ is projective,
and hence $\epsilon_i$ splits for any $i\in\bbZ_m$.
Setting
\[
A_i:=\left(\cdots \to 0 \to B^{i+1}(V)\to Z^{i+1}(V) \to 0 \to \cdots \right),
\]
where $Z^{i+1}(V)$ is in the degree $i$th,
we have $V \simeq \bigoplus_{0\le i \le m-1} A_i$ in $C_m(\catB)$.
Since $A_i \simeq H^i(V)[-i]$ in $D_m(\catB)$,
we have the conclusion.
\end{proof}

\begin{cor}
If $\catB$ is hereditary,
then the natural functor
$\pi : D^b(\catB) \to D_m(\catB)$
is dense.
\end{cor}
\begin{proof}
In this case, any periodic complex is a direct sum of stalk complexes.
Since any stalk complex is in the image of $\pi$,
we obtain the conclusion.
\end{proof}

\subsection{Periodic complexes as DG modules}\label{ss:periodic cpx as DG}
Let $\LL$ be an algebra.
We will explain that periodic complexes over $\LL$ can be considered as 
DG modules over the Laurent polynomial ring $\LL[t,t^{-1}]$ over $\LL$,
where $\LL[t,t^{-1}]$ is a graded algebra by setting $\deg(t):=m$.
Recall that a graded algebra is a DG algebra with the trivial differential.
\begin{lem}
Let $f:M\to N$ be a morphism in $C_{\dg}\left(\LL[t,t^{-1}]\right)$.
\begin{enua}
\item 
$M^i$ is a $\LL$-module for any $i\in \bbZ$.
\item
The right multiplication
\[
r_{t}: M \simto M,\quad m \mapsto m\cdot t
\]
induces an isomorphism $M^i\simto M^{i+m}$
of $\LL$-modules
for any $i\in \bbZ$.
\item
$d^{i+m}_M\circ r_t = r_t\circ d^i_M$ for any $i\in \bbZ$.
\item
$f\circ r_t= r_t\circ f$
\end{enua}
\end{lem}
\begin{proof}
(1) follows from 
the fact that the degree $0$th part of $\LL[t,t^{-1}]$ is equal to $\LL$.
$r_t$ is $\LL$-linear since $t$ is a central element,
and it restricts to $M^i \simto M^{i+m}$ for any $i\in\bbZ$
because $t$ has the degree $m$.
Hence (2) holds.
Since
\[
d(vt)=d(v)t +(-1)^{|v|}v d(t)=d(v)t
\]
for any $v\in M$,
we obtain (3).
(4) is clear.
\end{proof}

\begin{prp}
There exists a DG equivalence
$C_{\dg}\left(\LL[t,t^{-1}]\right) \simto C_m(\Mod\LL)_{\dg}$.
\end{prp}
\begin{proof}
Let $M\in C_{\dg}\left(\LL[t,t^{-1}]\right)$.
Define
\[
V(M)^{\ol{i}}:=M^i,\quad
d_{V(M)}^{\ol{i}}:=
\begin{cases}
d^i & i \ne m-1 \\
r_{t^{-1}} \circ d^{m-1} & i = m-1
\end{cases}
\]
Then $V(M):=\left(V(M)^{\ol{i}},d_{V(M)}^{\ol{i}} \right)_{0\le i \le m-1}$
is an $m$-periodic complex
since
\[
d^0 \circ r_{t^{-1}} \circ d^{m-1}=r_{t^{-1}} \circ d^m \circ d^{m-1}=0.
\]
This correspondence induces a DG equivalence between
$C_{\dg}\left(\LL[t,t^{-1}]\right)$ and $C_m(\Mod \LL)_{\dg}$.
\end{proof}

In particular, we have a triangulated equivalence
\[
D(\LL[t,t^{-1}]) \simto D_m(\Mod\LL), \quad \LL[t,t^{-1}] \mapsto \LL,
\]
since the above correspondence preserves acyclic.
In the reminder of this section,
we prove that the perfect derived category $\perf \LL[t,t^{-1}]$ 
corresponds to $D_m(\LL):=D_m(\catmod\LL)$ by the above equivalence.

\begin{fct}[{\cite{Zhao14}}]
Let $\LL$ be an algebra of finite global dimension.
Then
$D_m(\LL)$ is the triangulated hull of the orbit category $D^b(\LL)/\Sg^m$
of the bounded derived category $D^b(\LL)$ by $\Sg^m$.
\end{fct}

In particular, $D_m(\LL)$ is idempotent complete 
by the construction of the triangulated hull of an orbit category.
See \cite[\S 2.4]{Fu12} and \cite{Kel05} for details.

\begin{prp}\label{lem:perf}
Let $\LL$ be a $\bbk$-algebra of finite global dimension.
\begin{enua}
\item
$\thick_{D_m(\Mod \LL)}(\LL)=\ol{K_m(\proj \LL)}$,
where $\ol{K_m(\proj \LL)}$ is the essential image of 
the natural inclusion $K_m(\proj \LL) \inj D_m(\Mod \LL)$.
\item
The triangulated equivalence $D(\LL[t,t^{-1}])\simto D_m(\Mod\LL) $
induces a triangulated equivalence $\perf(\LL[t,t^{-1}]) \simto D_m(\LL)$.
\end{enua}
\end{prp}
\begin{proof}
Since $K_m(\proj\LL)$ is triangulated equivalent to $D_m(\LL)$,
$\ol{K_m(\proj\LL)}$ is a thick subcategory of $D_m(\Mod\LL)$.
Hence $\thick_{D_m(\Mod \LL)}(\LL) \subset \ol{K_m(\proj \LL)}$
as $\LL \in \ol{K_m(\proj \LL)}$.
By Corollary \ref{cor:proj periodic},
we have $K_m(\proj\LL)=\thick_{K_m(\proj\LL)}(\LL)$.
Thus we have $\ol{K_m(\proj\LL)} \subset \thick_{D_m(\Mod\LL)}(\LL)$.
Therefore we obtain (1).
(2) follows from (1).
\end{proof}

\section{Proof of the main theorem}\label{s:main}
Let $\catT$ be a triangulated category.
Recall that $\catT$ is said to be $m$-periodic
if its suspension functor $\Sg$ satisfies $\Sg^m\simeq \Id_{\catT}$ as additive functors.
Suppose $\catT$ has an $m$-periodic tilting object $T$ (Definition \ref{dfn:m-tilt})
and its endomorphism algebra $\LL :=\End_{\catT}(T)$
is a finite dimensional $\bbk$-algebra.
Since $T$ is a thick generator of $\catT$,
there exists a DG algebra $B$ and 
a triangulated equivalence $\catT \simeq \perf(B)\subset D(B)$
such that
$H^*(B)=\bigoplus_{i\in\bbZ}H^i(B)\simeq \bigoplus_{i\in\bbZ}\catT(T, \Sg^iT)$
by Fact \ref{fct:morita}.
By the definition of
$m$-periodic categories and $m$-periodic tilting objects,
we have an isomorphism of graded algebras
\[
H^*(B)
\simeq \bigoplus_{i\in\bbZ}\catT(T, \Sg^iT)
\simeq \bigoplus_{i\in m\bbZ}\catT(T, \Sg^iT)
\simeq \bigoplus_{i\in m\bbZ}\catT(T, T)
\simeq \bigoplus_{i\in m\bbZ}\LL
\simeq \LL[t,t^{-1}],
\]
where $\LL[t,t^{-1}]$ is the Laurent polynomial ring  over $\LL$
which is regard as a graded algebra by setting $\deg(t):=m$.
On the other hand, we have triangulated equivalences $D(\LL[t,t^{-1}]) \simto D_m(\Mod \LL)$ and $\perf(\LL[t,t^{-1}]) \simto D_m(\LL)$ by Lemma \ref{lem:perf}.
Hence we have the following diagram.
\[
\begin{tikzcd}
&D(B) \arr{r,"\sim","?"',dashed} & D(\LL[t,t^{-1}]) \arr{r,"\sim"} & D_m(\Mod \LL)\\
\catT \arr{r,"\sim"} & \perf(B) \arr{r,"\sim","?"',dashed} \arr{u,hook} &\perf(\LL[t,t^{-1}]) \arr{u,hook} \arr{r,"\sim"} & D_m(\LL). \arr{u,hook}
\end{tikzcd}
\]
If $B$ and $\LL[t,t^{-1}]$ are quasi-isomorphic as DG algebras,
there exists a triangulated equivalence $D(B) \simeq D(\LL[t,t^{-1}])$ 
by Lemma \ref{fct:DG qis},
and thus it induces a triangulated equivalences $\perf(B) \simto \perf(\LL[t,t^{-1}])$
between perfect derived categories.
Hence to prove $\catT \simeq D_m(\LL)$,
it is enough to show that $B$ is quasi-isomorphic to $\LL[t,t^{-1}]\simeq H^*(B)$,
that is, $B$ is a formal DG algebra (Definition \ref{dfn:dg:formal}).
Recall that
a sufficient condition of being formal is given by
the vanishing of the Hochschild cohomology of its cohomology algebra (\S \ref{ss:Hoc}).
Thus we now compute the Hochschild cohomology of $\LL[t,t^{-1}]$.

\begin{lem}\label{lem:hoch}
Let $\LL$ be an algebra which has a finite projective dimension $d$ as a $\LL$-bimodule.
We regard the Laurent polynomial ring $\LL[t,t^{-1}]$ over $\LL$ 
as a graded algebra by setting $\deg(t):=m$.
Let $p,q \in \bbZ$ and consider the Hochschild cohomology $HH^{p,q}(\LL[t, t^{-1}])$.
If $p\geq d+2$ or $q\not\equiv 0 \mod m$, then we have 
\[
 HH^{p,q}\left( \LL [t, t^{-1}] \right) = 0.
\]
\end{lem}

\begin{proof}
Let $A := \LL[t, t^{-1}]$.
Since $\pd_{\LL^e} \LL = d$,
we have a projective resolution
\[
F_{\bu}: 0\to F_d \to \dots \to F_0 \to \LL \to 0
\]
of $\LL$ of length $d$ as a $\LL^e$-module.
Let $\bbk[t,t^{-1}]$ be the Laurent polynomial ring over $\bbk$ with $\deg(t)=m$.
Then we have $S:= \bbk[t,t^{-1}]^e \simeq \bbk[x, x^{-1}, y, y^{-1}]$
as graded algebras with $\deg(x)=\deg(y)=m$.
As a graded $S$-module, $\bbk[t, t^{-1}]$ has a free resolution
\[
	G_{\bu}: 0 \to S(-m) \xr{f} S \xr{g} \bbk[t, t^{-1}] \to 0
\]
of length $2$,
where $f$ is a multiplication map defined by $x-y$ 
and $g$ is a substitution map defined by $x \mapsto t$ and $y \mapsto t$.
As a graded $A^e$-module,
$H_{\bu}:= F_{\bu} \otimes G_{\bu}$ is a projective resolution 
of $A= \LL \otimes \bbk[t, t^{-1}]$ of length $d+1$.
Calculating the Hochschild cohomology of $A$ 
by using the projective resolution $H_{\bu}$,
we have
\[
	HH^{p,q}(A)
	=\Ext^p_{\Gr A^e}(A, A(q))
	=H^p\left( \Hom_{\Gr A^e} (H_{\bu}, A(q))\right).
\]
Since length of $H_{\bu}$ is $d+1$,
$HH^{p,q}(A)=0$ for $p\geq d+2$.
Note that $H_{i}$ is concentrated in degrees $0$ and $m$ for all $i$,
and $\LL[t, t^{-1}]$ is concentrated in degrees of multiples of $m$.
For $q \not\equiv 0 \mod m$, $\Hom_{\Gr A^e} (H_{\bu}, A(q))=0$
and hence $HH^{p,q}(A)=0$.
\end{proof}

\begin{cor}\label{cor:formal}
Under the assumption of Lemma \ref{lem:hoch},
$\LL[t, t^{-1}]$ is intrinsically formal if $d \leq m$.
\end{cor}

\begin{proof}
By Lemma \ref{lem:hoch}, $HH^{q, 2-q}(\LL[t, t^{-1}])=0$
if $q\geq d+2$ or $q \not\equiv 2 \mod m$.
Since $m+2 \geq d+2$ by the assumption,
$HH^{q, 2-q}(\LL[t, t^{-1}])=0$ for $q \geq 3$.
Hence $\LL[t, t^{-1}]$ is intrinsically formal by Fact \ref{fct:if}.
\end{proof}

By the discussion so far, we have the following theorem.
\begin{thm}\label{thm:Main}
Let $\catT$ be an idempotent complete algebraic $m$-periodic triangulated category.
Suppose $\catT$ has an $m$-periodic tilting object $T$
and its endomorphism algebra $\LL :=\End_{\catT}(T)$
is a finite dimensional $\bbk$-algebra.
If $\pd_{\LL^e}\LL \le m$,
then there exists a triangulated equivalence
\[
\catT\simeq D_m\left( \Lambda \right).
\]
\qed
\end{thm}

In the reminder of this section,
we give a restatement of Theorem \ref{thm:Main} in a convenient form.
We first recall homologically smoothness of algebras.
We refer to \cite[Section 3]{RR20} for a detailed explanation.
\begin{dfn}
Let $\LL$ be an algebra.
\begin{enua}
\item $\LL$ is \emph{homologically smooth} of dimension $d$
if $\LL$ has finite projective dimension equal to $d$ as an $\LL^e$-module.
\item $\LL$ is \emph{separable} if $\LL$ 
is homologically smooth of dimension $0$.
\end{enua}

For finite dimensional algebras,
we have a useful characterization of homologically smoothness.
\end{dfn}
\begin{fct}[{\cite[Corollary 3.19, 3.22]{RR20}}]\label{fct:HS}
A finite dimensional algebra $\LL$ is homologically smooth of dimension $d$
if and only if
$\LL/\rad \LL$ is separable and $\gd \LL=d$
\end{fct}

Since semisimple algebras over a perfect field are separable
(c.f. \cite[Lemma 3.3]{RR20}),
we have next corollary.
\begin{cor}
A finite dimensional algebra $\LL$ over a perfect field $\bbk$
is homologically smooth of dimension $d$
if and only if
$\gd \LL=d$.
\end{cor}

We restate Theorem \ref{thm:Main} in a convenient form.
\begin{cor}\label{cor:Main}
Let $\catT$ be an idempotent complete algebraic 
$m$-periodic triangulated category over a perfect field $\bbk$.
If $\catT$ has an $m$-periodic tilting object $T$
whose endomorphism algebra $\LL :=\End_{\catT}(T)$
is a finite dimensional $\bbk$-algebra of global dimension $d\leq m$,
then there exists a triangulated equivalence
\[
\catT\simeq D_m\left( \Lambda \right).
\]
\end{cor}

\section{Examples}\label{s:ex}
In this section, 
we give concrete examples of periodic categories and periodic tilting objects
and an application of Corollary \ref{cor:Main}.
For a periodic triangulated category $\catT$,
the \emph{period} of $\catT$ is
the smallest positive integer $m$ with $\Sg^m \simeq \Id_{\catT}$ as additive functors.
\subsection{$m$-periodic derived categories}\label{ss:ex Dm}
Let $m\ge 1$ be a positive integer.
The $m$-periodic derived category $D_m(\LL)=D_m(\catmod \LL)$ of an algebra $\LL$
is a fundamental example of periodic triangulated category,
but its period is not necessarily $m$.
The period depends on the parity of $m$.
This phenomena is caused by the change of signs of differential by the shift functor.
\begin{prp}
Let $p$ be the period of $D_m(\LL)$.
\begin{enua}
\item
If $m$ or $\cha(\bbk)$ is even,
then $p=m$.
\item
If $m$ is odd,
then $p=m$ or $2m$.
\end{enua}
\end{prp}
\begin{proof}
Considering stalk complexes,
we conclude that $p$ is a multiple of $m$.
For an $m$-periodic complex $(V,d)$, $\Sg^{2m}(V,d)=\left(V,(-1)^{2m}d\right)=(V,d)$.
Hence $\Sg^{2m}=\Id_{D_m(\LL)}$.
This implies $p=m$ or $2m$.
If $m$ or $\cha(\bbk)$ is even,
$\Sg^m(V,d)=\left(V,(-1)^md\right)=(V,d)$.
Thus, in this case, $p=m$.
\end{proof}

There is an algebra $\LL$ 
whose $1$-periodic derived category $D_1(\LL)$ is 
not a $1$-periodic triangulated category \cite[\S 5.3]{St18}.
However, for hereditary algebras, we do not worry about this phenomena.
\begin{prp}
Let $\LL$ be a hereditary algebra.
Then the period of $D_m(\LL)$ is $m$.
\end{prp}
\begin{proof}
Let $\catD$ be a full subcategory of $D_m(\LL)$ consisting of
$m$-periodic complexes with trivial differential.
$\Sg^m(V,0)=(V,0)$ for any $(V,0)\in\catD$ implies
$(\Sg_{|\catD})^m=\Id_{\catD}$.
Since the natural inclusion functor $\catD \inj D_m(\LL)$ 
is an equivalence by Proposition \ref{prp:hered periodic},
$\Sg^m \simeq \Id_{D_m(\LL)}$ as additive functors.
\end{proof}

$m$-periodic derived categories have $m$-periodic tilting objects
if it is an $m$-periodic triangulated category.
\begin{prp}
Let $\LL$ be an algebra of finite global dimension.
Suppose that the period of $D_m(\LL)$ is $m$.
Then $\LL_{\LL}$ itself is an $m$-periodic tilting object of $D_m(\LL)$
\end{prp}
\begin{proof}
By Corollary \ref{cor:proj periodic} (1), $\LL_{\LL} \in D_m(\LL)$ is a thick generator.
Since $\LL$ is $K$-projective,
we have
\[
D_m(\LL)(\LL,\LL[i])\simeq K_m(\LL)(\LL,\LL[i])
=\begin{cases}
\LL & i\in m\bbZ \\
0 & i\not\in m\bbZ.
\end{cases}
\]
\end{proof}
This proposition indicates the possibility 
that we can drop the assumption $\pd_{\LL^e}\LL \leq m$ in Theorem \ref{thm:Main}.
On the other hand,
Example \ref{ex:cm} shows that we cannot drop the assumption $\pd_{\LL^e}\LL < \infty$.

\subsection{Periodic algebras}\label{ss:ex periodic alg}
In this subsection,
we give a quick review on periodic algebras following \cite[\S 11]{SY11}
and explain the relation between periodic algebras and periodic triangulated categories.
As an application of Corollary \ref{cor:Main},
we give a nontrivial triangulated equivalence between 
the stable category of the category of modules over a periodic algebra
and the periodic derived category of a hereditary algebra
in Example \ref{ex:self-inj Nalayama}.

Let $\LL$ be an algebra and $\LL^e:=\LL^{\op}\otimes_{\bbk}\LL$ its enveloping algebra.
For a finitely generated $\LL$-module $M$,
$P(M)\surj M$ denotes the projective cover of $M$
and $\Omega^1_{\LL}(M):=\Ker(P(M)\surj M)$ denotes the $1$st syzygy of $M$.
Inductively, we define the $n$th syzygy of $M$ by
$\Omega^n_{\LL}(M):= \Omega^1_{\LL}\left(\Omega^{n-1}_{\LL}(M)\right)$.
\begin{dfn}
An algebra $\LL$ is \emph{periodic}
if $\Omega^n_{\LL^e}(\LL)\simeq \LL$ in $\catmod \LL^e$.
The \emph{period} of $\LL$ is the smallest positive integer $m$ with
$\Omega^n_{\LL^e}(\LL)\simeq \LL$ in $\catmod \LL^e$.
\end{dfn}

\begin{ex}[{\cite[Lemma 4.2]{EH99}}]\label{ex:EH99}
Let $Q_n$ be the cyclic quiver in Figure \ref{fig:cyc}.
\begin{figure}[H]
\centering
\begin{tikzpicture}
	\coordinate[label=90:$1$] (1) at (90:2);
	\coordinate[label=45:$2$] (2) at (45:2);
	\coordinate[label=0:$3$] (3) at (0:2);
	\coordinate[label=270:$i$] (i) at (270:2);
	\coordinate[label=180:$n-1$] (n-1) at (180:2);
	\coordinate[label=135:$n$] (n) at (135:2);
	 \foreach \P in {1,2,3,i,n-1,n} \draw (\P) circle (0.06);
	 \draw[->,shorten <=7pt,shorten >=7pt] (1) arc (90:135:2);
	 \draw[->,shorten <=7pt,shorten >=7pt] (2) arc (45:90:2);
	 \draw[->,shorten <=7pt,shorten >=7pt] (3) arc (0:45:2);
	 \draw[<-,shorten <=7pt,shorten >=7pt] (3) arc (360:320:2);
	 \draw[->,shorten <=7pt,shorten >=7pt] (i) arc (270:310:2);
	 \draw[<-,shorten <=7pt,shorten >=7pt] (i) arc (270:230:2);
	 \draw[->,shorten <=7pt,shorten >=7pt] (n-1) arc (180:220:2);
	 \draw[->,shorten <=7pt,shorten >=7pt] (n) arc (135:180:2);
	 \draw[thick,dotted] (310:2) arc (310:320:2);
	 \draw[thick,dotted] (220:2) arc (220:230:2);
\end{tikzpicture}
\caption{The cyclic quiver $Q$.}
\label{fig:cyc}
\end{figure}
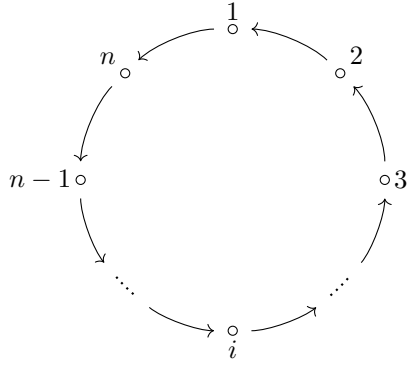
\noindent
We identify the numbers assigned to vertices of $Q$ with the elements of $\bbZ_n$ 
as in Figure \ref{fig:cyc}.
The algebra $\LL_{n,m}:=\bbk Q_n/\rad^m(\bbk Q_n)$ is a periodic algebra.
The period of $\LL_{n,m}$ is $2\cdot \lcm(n,m)/m$
except if $\cha(\bbk)=2$, $m=2$ and $n$ is odd.
In the excluded case, the period of $\LL_{n,m}$ is $n$.
\end{ex}

The stable category of $\smod \LL$ is the quotient category of $\catmod \LL$ by
the ideal generated by the morphisms factoring through projective modules.
The $n$th syzygy gives rise to a functor $\Omega^n_{\LL}(-) : \smod\LL \to \smod\LL$
\begin{fct}[{\cite[Proposition 11.17]{SY11}}]\label{fct:bimod syzygy}
Let $\LL$ be an algebra and $M\in\catmod\LL$.
There exists a natural isomorphism
\[
\Omega^i_{\LL}(M)\simeq M\otimes_{\LL}\Omega^i_{\LL^e}(\LL)
\]
in $\smod \LL$ for any $i\ge 0$.
\end{fct}

\begin{prp}
Let $\LL$ be a periodic algebra.
\begin{enua}
\item
$\LL$ is a self-injective algebra.
In particular, $\smod\LL$ is a triangulated category whose suspension functor is
a quasi-inverse of syzygy functor $\Omega^1_{\LL}(-):\smod\LL \to \smod\LL$
\item
$\smod \LL$ is a periodic triangulated category.
\item
The period of $\smod\LL$ divides the period of $\LL$.
\end{enua}
\end{prp}
\begin{proof}
See \cite[Proposition 11.18]{SY11} for (1).
Let $d$ be the period of $\LL$.
Then $\Omega^d_{\LL^e}(\LL)\simeq \LL$ as $\LL^e$-modules.
By Fact \ref{fct:bimod syzygy}, we have an isomorphism
\[
\Omega^d_{\LL}(M)
\simeq M\otimes_{\LL}\Omega^d_{\LL^e}(\LL)
\simeq M\otimes_{\LL}\LL 
\simeq M
\]
in $\smod \LL$.
It is obviously natural isomorphism.
Hence $\Omega^d_{\LL}\simeq \Id_{\smod\LL}$.
Thus we conclude (2) and (3).
\end{proof}

We investigate the stable category $\smod\LL_{n,m}$ 
in Example \ref{ex:EH99} for the case $n=m$, explicitly.
We construct an equivalence between $\smod\LL_{n,n}$
and the periodic derived category of a hereditary algebra by Corollary \ref{cor:Main}.
\begin{ex}\label{ex:self-inj Nalayama}
Let $\bbk$ be a perfect field.
Consider the algebra $\LL :=\LL_{n,n}$ in Example \ref{ex:EH99}.
$\LL$ is a periodic algebra whose period is $2$.
Thus the period of $\smod\LL$ is $1$ or $2$.
The Auslander-Reiten (AR) quiver of $\LL$ is Figure \ref{fig:selfinj},
where $M(a, l)$ is the indecomposable $\LL$-module of length $l$
whose top is the simple $\LL$-module associated to the vertex $a$.
\begin{figure}[H]
\centering
\begin{tikzpicture}
\node (1) at (0,0) {$M(1,1)$};
\node (2) at (3,0) {$M(2,1)$};
\node (3) at (6,0) {$M(3,1)$};
\node (21) at (1.5,1) {$M(2,2)$};
\node (32) at (4.5,1) {$M(3,2)$};
\node (13) at (7.5,1) {$M(1,2)$};
\node (321) at (3,2) {$M(3,3)$};
\node (132) at (6,2) {$M(1,3)$};
\node (213) at (9,2) {$M(2,3)$};
\node (1') at (9,0) {$M(1,1)$};
\node (21') at (10.5,1) {$M(2,2)$};
\node (A) at (-3,0) {};
\node (B) at (-1.5,1) {};
\node (C) at (0,2) {};
\node (X) at (1.5,2) {};
\node (Y) at (-1.5,0) {};
\node (Z) at (10.5,2) {};
\node (W) at (7.5,0) {};
 \draw[dashed] (1)--(2);
 \draw[dashed] (2)--(3);
 \draw[dashed] (21)--(32);
 \draw[dashed] (32)--(13);
 \draw[->] (1)--(21);
 \draw[->] (21)--(2);
 \draw[->] (2)--(32);
 \draw[->] (32)--(3);
 \draw[->] (3)--(13);
 \draw[->] (21)--(321);
 \draw[->] (321)--(32);
 \draw[->] (32)--(132);
 \draw[->] (132)--(13);
 \draw[->] (13)--(213);
 \draw[dashed] (13)--(21');
 \draw[dashed] (3)--(1');
 \draw[->] (213)--(21');
 \draw[->] (13)--(1');
 \draw[->] (1')--(21');
 \draw[dashed,shorten <=45pt] (A)--(1);
 \draw[dashed,shorten <=45pt] (B)--(21);
 \draw[->,shorten <=23pt] (B)--(1);
 \draw[->,shorten <=23pt] (C)--(21);
 \draw[dotted,shorten <=-10pt,shorten >=-10pt] (X)--(Y);
 \draw[dotted,shorten <=-10pt,shorten >=-10pt] (Z)--(W);
\end{tikzpicture}
\caption{The AR quiver of $\LL$ for $n=3$.}
\label{fig:selfinj}
\end{figure}
The indecomposable projective modules are $M(a,n)$ for vertices $a$ of $Q_n$.
Since the sequence $0\to M(a+l, n-l) \to M(a+l , n) \to M(a, l) \to 0$ is exact,
we have $\Sg M(a, l)= M(a+l, n-l)$.
Thus the period of $\smod \LL$ is $2$.
Then for a vertex $a$ of $Q_n$, the object 
\[
 T(a) := \bigoplus_{l=1}^{n-1} M(a, l) 
\]
is a $2$-periodic tilting object of $\smod\LL$.
Indeed, $T(a)$ is a rigid thick generator of $\smod\LL$ by Figure \ref{fig:selfinj}.
Figure \ref{fig:selfinj} also says that the endomorphism algebra of $T(a)$
is isomorphic to a hereditary algebra $\bbk A_{n-1}$,
where $A_{n-1}$ is the quiver $1 \leftarrow 2 \leftarrow \dots \leftarrow n-1$.
By Corollary \ref{cor:Main}, 
there exists a triangulated equivalence $\smod \LL \simto D_2(\bbk A_{n-1})$.
\end{ex}

\subsection{Hypersurface singularities}\label{ss:ex Hyper sing}
In this subsection, we investigate periodic tilting objects 
in the stable category of the category maximal Cohen-Macaulay (CM) modules
over a hypersurface singularity.
These categories are $2$-periodic triangulated categories and,
in some case, have a $2$-periodic tilting object.
However, its endomorphism algebra has infinite global dimension
and hence does not satisfies the assumption of Corollary \ref{cor:Main}.
In Example \ref{ex:cm},
we see that there is a $2$-periodic triangulated category 
with a $2$-periodic tilting object,
which is not triangulated equivalence to
the $2$-periodic derived category of its endomorphism algebra.

We first recall that the basics of hypersurface singularities.
See \cite{Yo90} for details.
Let $S$ be a $(d+1)$-dimensional complete regular local commutative ring 
with maximal ideal $\pnn$ and residue field $\bbk$.
For $0\ne f\in\pnn$,
$R:= S/(f)$ is called the \emph{hypersurface singularity} defined by $f$.
Note that $R$ is a $d$-dimensional complete Gorenstein commutative ring.
In particular, the category $\CM (R)$ of maximal CM modules over $R$
is a Frobenius exact category,
and thus its stable category $\sCM(R)$ is a triangulated category
whose shift functor is given by 
a quasi-inverse of syzygy functor $\Omega^1:\sCM(R) \to \sCM(R)$. 
By the theory of matrix factorizations,
we can see that $\sCM(R)$ is a $2$-periodic triangulated category.

A commutative local ring $R$ is called an isolated singularity
if the localization $R_{\ppp}$ is regular
for any non maximal prime ideal $\ppp$ of $R$.
In the following, we denote by $\sHom_R(M,N):=\Hom_{\sCM(R)}(M,N)$.
\begin{fct}[{\cite[\S 3]{Yo90}}]
Let $R$ be a complete Gorenstein isolated singularity
with maximal ideal $\pmm$ and residue field $\bbk$.
\begin{enua}
\item
$\CM(R)$ has Auslander-Reiten (AR) sequences.
\item
The AR translation is given by $\tau=\Omega^{2-d}$.
\item
$\sCM(R)$ is $(d-1)$-Calabi-Yau, i.e.
For any $M,N\in\CM(R)$, there are the natural isomorphism
\[
\sHom_R(M,N)\simeq \bbD\left( \sHom_R(N,\Omega^{1-d}M) \right),
\]
where $E$ is the injective envelope of $\bbk$ and 
$\bbD := \Hom_R(-,E):(\catmod R)^{\op} \to \catmod R$ is the Matlis dual.
\end{enua}
\end{fct}

Let $R = S/(f)$ be a hypersurface singularity.
Recall that an $R$-module $M$ is rigid if $\Ext^1_R(M, M) = 0$.
$2$-periodic tilting objects in $\sCM(R)$ are exactly non-free rigid CM $R$-modules,
which are thick generators of $\sCM(R)$.
The following proposition is suggested by Ryo Takahashi.
\begin{prp}\label{prp:Tak}
Let $R=S/(f)$ be a hypersurface isolated singularity.
\begin{enua}
\item
If $\dim R=2$,
$\sCM(R)$ has no $2$-periodic tilting objects.
\item
Any non-free rigid CM $R$-modules are $2$-periodic tilting objects.
\end{enua}
\end{prp}
\begin{proof}
(1)
Since $\sCM(R)$ is $1$-Calabi-Yau,
$\Ext^1_R(M,M)=0$ implies $\sHom_R(M,M)=0$.
Thus rigid CM $R$-modules isomorphic to a zero module in $\sCM(R)$.

(2)
It is enough to show that non-free CM $R$-modules are thick generators.
In this case, there are no nontrivial thick subcategory
by \cite[Theorem 6.8.]{Tak10}.
Thus non-free CM $R$-modules are thick generators.
\end{proof}

By this proposition,
we can determine
whether the stable category of maximal CM modules over a simple singularity
has a $2$-periodic tilting object.
Recall that \emph{simple singularities} are 
hypersurface singularities $\bbC[[x,y,z_2,\dots,z_d]]/(f)$ 
defined by the following defining equations $f\in \bbC[[x,y,z_2,\dots,z_d]]$:
\begin{description}
\item[($A_n$)]
$x^2+y^{n+1}+z_2^2 +\cdots + z_d^2$ for $n\ge 1$,
\item[($D_n$)]
$x^2y+y^{n-1}+z_2^2 +\cdots + z_d^2$ for $n\ge 4$,
\item[($E_6$)]
$x^3+y^{4}+z_2^2 +\cdots + z_d^2$,
\item[($E_7$)]
$x^3+xy^{3}+z_2^2 +\cdots + z_d^2$,
\item[($E_6$)]
$x^3+y^{5}+z_2^2 +\cdots + z_d^2$.
\end{description}

\begin{prp}
Let $R_d=\bbC[[x,y,z_2,\dots,z_d]]/(f)$ be a simple singularity.
\begin{enua}
\item
Assume $d$ is even.
Then $\sCM(R_d)$ has no $2$-periodic tilting objects.
\item
Assume $d$ is odd.
Then $\sCM(R_d)$ has a $2$-periodic tilting object
if $R_d$ is of type ($A_n$) for odd $n$, ($D_n$) or ($E_7$).
In any case, the endomorphism algebra of a $2$-periodic algebra
has infinite global dimension.
\end{enua}
\end{prp}
\begin{proof}
By the Kn\"{o}rrer periodicity \cite[Theorem 3.1]{Kn87},
$\sCM(R_d)$ is triangulated equivalent to $\sCM(R_2)$ or $\sCM(R_1)$
if $d$ is even or odd, respectively,
where $R_1=\bbC[[x,y]]/(x^2+y^{n+1})$, for example.
By Proposition \ref{prp:Tak} (1), $\sCM(R_2)$ has no $2$-periodic object,
and hence $\sCM(R_d)$ has no $2$-periodic object for even $d$.
Assume $d$ is odd.
Then $2$-periodic tilting objects in $\sCM(R_1)$ is equivalent to
non-free rigid CM $R_1$-modules by Proposition \ref{prp:Tak} (1).
Non-free rigid CM $R_1$-modules are determined in \cite[Theorem 1.3]{BIKR08},
and hence (2) follows from this.
\end{proof}

We write down explicitly the stable category of maximal CM modules 
over a simple singularity $R_{2\ell-1}$ of type $(A_{2\ell-1})$ for $\ell\ge1$
and see that $\sCM(R_{2\ell-1})$ is not equivalent to the $2$-periodic derived category
of the endomorphism algebra of its $2$-periodic object.
\begin{ex}\label{ex:cm}
Consider the one dimensional simple singularity
$R_{2\ell-1} := \bbC[[x,y]] /\left( x^2 - y^{2\ell} \right)$ of type $A_{2\ell-1}$.
The AR quiver of $\sCM(R_{2\ell-1})$ is given as in Figure \ref{fig:A_2l-1}.
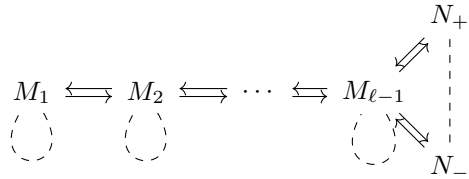
\begin{figure}[H]
\centering
\begin{tikzpicture}
\node (1) at (0,0) {$M_1$};
\node (2) at (1.5,0) {$M_2$};
\node (i) at (3,0) {$\cdots$};
\node (n-1) at (4.5,0) {$M_{\ell-1}$};
\node (+) at (5.5,1) {$N_{+}$};
\node (-) at (5.5,-1) {$N_{-}$};
 \draw[->,shorten <=2pt,shorten >=2pt,transform canvas={yshift=-1.5pt}]
 (1) to (2);
 \draw[->,shorten <=2pt,shorten >=2pt,transform canvas={yshift=1.5pt}]
 (2) to (1);
 \draw[->,shorten <=2pt,shorten >=2pt,transform canvas={yshift=-1.5pt}]
 (2) to (i);
 \draw[->,shorten <=2pt,shorten >=2pt,transform canvas={yshift=1.5pt}]
 (i) to (2);
 \draw[->,shorten <=2pt,shorten >=2pt,transform canvas={yshift=-1.5pt}]
 (i) to (n-1);
 \draw[->,shorten <=2pt,shorten >=2pt,transform canvas={yshift=1.5pt}]
 (n-1) to (i);
 \draw[->,shorten <=2pt,transform canvas={yshift=1.75pt}]
 (+) to (n-1);
 \draw[->,shorten <=2pt,transform canvas={yshift=-1.75pt}]
 (n-1) to (+);
 \draw[->,shorten <=2pt,transform canvas={yshift=1.75pt}]
 (n-1) to (-);
 \draw[->,shorten <=2pt,transform canvas={yshift=-1.75pt}]
 (-) to (n-1);
 \draw[shorten <=1.5pt,dashed]
 (-) to (+);
 \draw[every loop/.style={},dashed]
 (1) edge  [looseness=2.5,in=240,out=300,loop] (1);
 \draw[every loop/.style={},dashed]
 (2) edge  [looseness=2.5,in=240,out=300,loop] (2);
 \draw[every loop/.style={},dashed]
 (n-1) edge  [looseness=2.5,in=240,out=300,loop] (n-1);
\end{tikzpicture}
\caption{The stable AR quiver of $A_{2\ell-1}$.}
\label{fig:A_2l-1}
\end{figure}
Then the maximal Cohen-Macaulay $R_{2\ell-1 }$-modules 
$N_{\pm}:=\bbC[[x, y]]/\left( x \pm y^{\ell} \right)$
are $2$-periodic tilting objects of $\sCM(R_{2\ell-1})$, but the endomorphism algebras
of $N_{\pm}$ are both isomorphic to $\bbC[x]/\left(x^{\ell} \right)$,
which has infinite global dimension \cite[Proposition 2.4]{BIKR08}.

For the case of $\ell=2$,
$\sCM(R_3)$ has $3$ indecomposable objects.
On the other hand, $D_2(\bbC[x]/(x^2))$ has more than $4$ indecomposable objects,
namely, stalk complexes $\bbC[x]/(x^2), (x)/(x^2) \in D_2(\bbC[x]/(x^2))$ and its shifts.
Thus $\sCM(R_3)$ and $D_2(\bbC[x]/(x^2))$ are not equivalent as additive categories.
In particular,
the Laurent polynomial ring $(\bbC[x]/(x^2))[t,t^{-1}]$ with $\deg (t)=2$
is not intrinsically formal.
\end{ex}


\end{document}